\documentclass[11pt]{amsart}
\usepackage{amsmath}
\usepackage{amssymb}
\usepackage{amscd}
\usepackage{url}
\usepackage{pstricks}
\usepackage{pst-node}
\usepackage{pst-coil}
\usepackage{graphicx}

\newcommand{\CC}{\mathbb{C}}

\newcommand{\bp}{\begin{pmatrix}}
\newcommand{\ep}{\end{pmatrix}}
\newcommand{\ol}[1]{\overline{#1}}
\newcommand{\LS}{L^{\text{sing}}}
\newcommand{\nn}{\widetilde{n}}
\newcommand{\cp}{\Delta^h}

\newcommand{\normal}[1]{{#1}^\wedge}

\DeclareMathOperator{\lk}{lk}

\DeclareMathOperator{\sign}{signature}

\DeclareMathOperator{\Fib}{Fib}
\DeclareMathOperator{\MHS}{MHS}
\DeclareMathOperator{\HVS}{HVS}
\DeclareMathOperator{\Irr}{Irr}
\DeclareMathOperator{\Id}{Id}
\newcommand{\fbi}{\Fib_\infty}

\numberwithin{equation}{section}
\numberwithin{equation}{subsection}

\theoremstyle{plain}
\newtheorem{theorem}[equation]{Theorem}
\newtheorem{lemma}[equation]{Lemma}
\newtheorem{proposition}[equation]{Proposition}

\newtheorem{corollary}[equation]{Corollary}

\theoremstyle{definition}
\newtheorem{example}[equation]{Example}
\newtheorem{remark}[equation]{Remark}

\newtheorem{definition}[equation]{Definition}

\numberwithin{equation}{subsection}

\newtheorem*{acknowledgements}{Acknowledgements}

\newcommand{\mv}{\mathcal{V}}

\newcommand{\mw}{\mathcal{W}}

\newcommand{\nd}{\textrm{ndeg}}
\newcommand{\reg}[1]{{#1}^\infty_{reg}}

\def\Z{\mathbb Z}\def\Q{\mathbb Q}

\newcommand{\tY}{\widetilde{Y}_0}

\title[Spectrum of plane curves]{Spectrum of plane curves via knot theory}
\author{Maciej Borodzik}
\address{Institute of Mathematics, University of Warsaw, ul. Banacha 2,
02-097 Warsaw, Poland}
\email{mcboro@mimuw.edu.pl}
\thanks{The first author is supported by Polish MNiSzW Grant No N N201 397937 and also by a Foundation for Polish Science FNP.
The second author is partially supported by OTKA Grant K67928. }
\author{Andr\'as N\'emethi}
\address{A. R\'enyi Institute of Mathematics, 1053 Budapest,  Re\'altanoda u. 13-15,  Hungary.}
\email{nemethi@renyi.hu}

\date{\today}
\makeatletter
\@namedef{subjclassname@2010}{\textup{2010} Mathematics Subject Classification}
\makeatother

\begin{document}
\begin{abstract}
In this paper we use topological methods to study various
semicontinuity properties of spectra of singular points of plane algebraic curves and
of polynomials in two variables at infinity.
Using the Seifert form and the Tristram--Levine signature of links,
 we reprove (in a slightly weaker version)
 a result obtained by Steenbrink and Varchenko on semicontinuity of spectra of singular points under deformation
and results of N\'emethi and Sabbah on semicontinuity of spectrum at infinity. We also relate the spectrum at infinity of a polynomial with spectra
of singular points of a chosen fiber.
\end{abstract}
\maketitle

\section{Introduction}
The Hodge spectrum of a local hypersurface isolated singularity $f:(\CC^{n+1},0)\to(\CC,0)$
is the output of the
mixed Hodge structure of the vanishing cohomology of the singular germ
 \cite{A,Stee,St,Var,Var2}. Usually, it is not
topological, it is one of the finest analytic invariants of the germ.
Although, it does not characterize the singularity completely, it gives
extremely strong  information about it. As it was conjectured by Arnold \cite{A},
 and proved by Varchenko \cite{Var,Var2} and Steenbrink \cite{St},
the spectrum behaves semicontinuously under deformations, which makes it, for example,
a very strong tool in attempts to solve the adjacency problem (i.e., to determine,
which singularities can specialize to a given one).

A more precise picture is the following: the
algebraic monodromy acts on the vanishing cohomology, this
cohomology supports a mixed Hodge structure, which is polarized by
the intersection form, and the Seifert form (which can be identified with the variation map).
The equivariant  Hodge numbers were  codified by Steenbrink  in the spectral pairs; if one deletes the
information about the weight filtration, one gets the spectral numbers $Sp(f)$. They are (in some normalization)
rational numbers in the interval $(0,n+1)$. In the presence of a deformation $f_t$,
where $t$ is the deformation parameter $t\in(\CC,0)$, the semicontinuity  guarantees that
 $|Sp(f_0)\cap I|\geq |Sp(f_{t\not=0})\cap I|$ for the semicontinuity domain $I$.
 Arnold made his conjecture for $I=(-\infty,x]$, Steenbrink and Varchenko proved the statement
 for $I=(x,x+1]$, which implies Arnold's conjecture. Additionally, Varchenko in \cite{Var} for some
 cases verified   the strongest version, namely semicontinuity for $I=(x,x+1)$.

The semicontinuity  property (with any domain)
cannot be extended to the spectral pairs, hence in studies targeting these kind of
applications one usually  works with spectrum only. This is what we will do in the present article as well.

On the other hand, (one of) the strongest topological invariants of $f$ is its Seifert form, for terminology see e.g.
\cite{AGV}.
The relation between the Hodge invariants and Seifert form was established by the second author in \cite{Nem2},
proving that the collection of  mod 2 spectral pairs are equivalent with the real Seifert form.
In this way, the real Seifert form is in strong relationship with the mod 2 spectrum, that is with
the collection of numbers $x$ mod 2 in $(0,2]$, where $x$ run over $Sp(f)$.
Clearly, for plane curve singularities, i.e. when $n=1$, by taking mod 2 reduction we loose no information.

\vspace{2mm}

Our primary  goal is to  {\it extend the above correspondence for an arbitrary link $(S^3_R,L$)}, where
$S^3_R$ is the boundary of some ball with radius $R$ in $\CC^2$, and $L$ is the intersection of $S^3_R$ with
some affine algebraic curve $C$ in $\CC^2$. The primary interest is the {\it link at infinity} of such affine curve
(hence $R\gg 0$),
but we also wish to develop a method to study  any general $(S^3_R,L)$, for which the available methods
in the literature are rather  sparse.

Let us consider a complex polynomial map $F:\CC^2\to \CC$.
For its topology at infinity see Neumann's article \cite{Neu1}.
Our first main result {\it recovers the spectrum at infinity associated with the limit mixed Hodge structure
at infinity (supported by the cohomology of the generic fiber) from the real Seifert form of the
regular link at infinity associated with $F$}. In particular, we reobtain the spectrum at infinity
topologically, in  pure link--theoretical language.

\vspace{2mm}

The key bridge which connects the link--theoretical language and invariants with the Hodge theoretical
spectrum is the Tristram--Levine signature \cite{Tr,Le}.
For example, for the  weighted homogeneous singularity given
by $\{x^p-y^q=0\}$ with $p$ and $q$ relative prime integers, the spectrum is
$Sp_{p,q}=\{\frac{i}{p}+\frac{j}{q},\, 1\le i\le p-1,\,1\le j\le q-1\}$,
while the Tristram--Levine signature function of the $(p,q)-$torus knot, evaluated at $e^{2\pi ix}$ with $x\in(0,1)$, $pqx\not\in\mathbb{Z}$,
is equal to
$2|Sp_{p,q}\cap(x,x+1)|-(p-1)(q-1)$, see e.g. \cite{Li}.
In \cite{BN} we made this relation rigorous, showing a direct
translation between spectra of singularities and Tristram--Levine signatures of their links.

\vspace{2mm}

In this correspondence, what is really surprising --- and this is the  {\it seconds main message} of the
article --- is the fact that the {\it semicontinuity of the mod 2 spectrum is topological}:
it  can be recovered independently
of analytic (Hodge theoretical) tools, it follows from pure link theory.
More precisely, we prove that length one `intervals' intersected by the mod 2 spectrum, namely sets of type
$Sp\cap (x,x+1)$ and $(Sp\cap (0,x))\cup (Sp\cap (x+1,2])$, for $x\in [0,1]$,
satisfy semicontinuty properties, whenever  this is question is well--posed.

In this article we exemplify this by three cases: we recover the semicontinuity
(in the above form, with slight assumptions)
for deformations of local plane curve singularities, corresponding to results of Varchenko and Steenbrink,
and also  we establish a semicontinuity
of the spectrum at infinity associated with a family of polynomials in two variables, in the spirit of \cite{NS}.
The third case targets a new phenomenon:
in the context of an affine curve $C\subset \CC^2$
we show a semicontinuity connecting the local spectra of the singularities of $C$ with the spectrum at infinity
of $C$.
In all these cases,
the key link--theoretical ingredient is a Murasugi type inequality,
which controls the modification
of the  Tristram--Levine signature under those type of surgeries which appears when we pass from
$C\cap S^3_{r}$ to $C\cap S^3_{R}$ via Morse theory ($r<R$). This was studied
by the first author in \cite{Bo}.

\vspace{2mm}

The organization of the paper is the following.
In section 2 we review the theory of hermitian variation structures from \cite{Nem2}, their relation
with the spectrum,  how can one associate such a structure to a link \cite{BN}, and how it connects the spectrum with the
Tristram--Levine signatures \cite{BN}. We also recall some of the main results of \cite{Bo}
 about surgery inequalities of links of type $S^3_R\cap C$.
 Section 3 contains the study of the spectrum at infinity of a polynomial map
in terms of the Seifert form at infinity.  In section 4 we prove semicontinuity results regarding the spectrum.

\vspace{2mm}

For a finite set $A$, we denote by $|A|$ the cardinality of $A$.


\begin{acknowledgements}
The main part of the paper was accomplished during the stay of the first
author at Alfr\'ed R\'enyi Institute in Budapest. The first author expresses
his thank for the hospitality. His stay was supported by Foundation for
Polish Science (program KOLUMB) and by Polish MNiSzW Grant
No N N201 397937.
\end{acknowledgements}

\section{Variations structures of links}\label{S:var}
We recall in \ref{ss:hvs}  the definition of an abstract  {\it hermitian variation structure} and its {\it spectra},
and in \ref{ss:var} the definition of the hermitian variation structure and spectra associated with
links in a three-sphere. Subsection \ref{ss:MHS-def} reviews
the definition of mixed Hodge structures
and their Hodge spectra. Finally, in \ref{ss:MHS-TL}
we draw a relationship observed in \cite{BN} between spectra and Tristram--Levine signatures of links. In
\ref{S:morseplane} we recall some results from \cite{Bo} which are
crucial ingredients in the proof of the semicontinuity results of the last section.

\subsection{}\label{ss:hvs}{\bf  Hermitian variation structures}
 were introduced in   \cite{Nem2}, they are generalization of
$\varepsilon$--symmetric isometric structures. Here we review the minimal basics, for more details see \cite{Nem2,Nem-Var}.

Recall that a
structure $(U={\mathbb C}^n; b, h)$, where $b$ is an $\varepsilon$--symmetric hermitian  form on $U$
preserved by the automorphism
$h$ of $U$, is called an \emph{isometric structure} (for $\varepsilon=\pm 1$).
The classification of  isometric structures when $b$ is non--degenerate was established by Milnor  \cite{Milnor-forms}
(see also \cite{Neu-inv,Neu-spli}).
Any $\varepsilon$--hermitian variation structure (in short $\varepsilon$--HVS)
 can be regarded as the isometric structure together with an operator $V\colon U^*\to U$ such that
\begin{equation}\label{eq:complete}
\ol{V^*}=-\varepsilon V\ol{h^*}\ \text{ and } \ V\circ \tilde{b}=h-\Id,
\end{equation}
where $\tilde{b}$ is the form $b$ regarded as map from $U$ to $U^*$.
We denote it as $\mv=(U;b,h,V)$. Here $*$ denotes the duality, while \ $\overline{\cdot}$ \ the
complex conjugation.

\begin{definition}\label{def:completed}
We say that the isometric structure $(U;b,h)$ can be completed to a hermitian variation structure,
if there exists $V:U^*\to U$ such that \eqref{eq:complete}
is satisfied.
\end{definition}

If $b$ is non--degenerate, then  the isometric structure can be uniquely completed to a
HVS:  $V=(h-\Id)\circ \tilde{b}^{-1}$. In general, not every isometric structures can be
completed (see e.g. (\ref{prop:SS})(c) below). Moreover, if a completion exists,
in general, it is not unique (even if we restrict ourselves to non--degenerate matrices $V$, see e.g.
\cite[(2.7.7)]{Nem2}).

A HVS is called {\it simple} if $V$ is an isomorphism. The classification of simple HVS's is established
 in \cite{Nem2}.   Each simple variation structure is a direct sum of indecomposable simple
variation structures. Indecomposable
structures can be listed:   for each positive integer $k$, and for each
$\lambda\in\mathbb{C}$ such that $0<|\lambda|\le 1$ we have

\begin{itemize}
\item for $|\lambda|<1$ a unique simple indecomposable variation structure $\mv^{2k}_\lambda$;
\item for $|\lambda|=1$ two simple indecomposable structures, denoted by $\mw^k_\lambda(+1)$ and $\mw^k_\lambda(-1)$.
\end{itemize}

This classification is a refinement of the Jordan block decomposition of the matrix $h$
(or of Milnor's classification of non--degenerate isometric structures).   More precisely, the matrix $h$
corresponding to $\mw^k_\lambda(\pm 1)$ is a single Jordan block of size $k$ and eigenvalue $\lambda$,
while the one corresponding to $\mv^{2k}_\lambda$ has two Jordan blocks of size $k$: one with eigenvalue $\lambda$, the other
with eigenvalue $1/\lambda$. For their precise form see \cite{Nem2}.

Write  a simple variation structure $\mv$ as  a (unique)
sum of the indecomposable ones:

\begin{equation}\label{eq:mvlsum}
\mv=\bigoplus_{\substack{0<|\lambda|<1\\ k\ge 1}} q^k_\lambda\cdot   \mv^{2k}_\lambda
\oplus
\bigoplus_{\substack{|\lambda|=1\\ k\ge 1, \ u=\pm 1}} p^k_\lambda(u)\cdot  \mw^{k}_\lambda(u)
\end{equation}
for certain non--negative integers $q^k_\lambda$ and $p^k_\lambda(u)$.
Here we write $m\cdot \mv$ for $\mv\oplus\dots\oplus\mv$ ($m$-times).

The numbers $\{q^k_\lambda\}_{|\lambda|<1}$ and $\{p^k_\lambda(\pm 1)\}_{\lambda\in S^1}$ are called
the \emph{H--numbers} of the HVS $\mv$.

Using H--numbers we can define the {\it spectrum of $\mv$}. Sometimes, in order to emphasize the source of the definition,
we call it {\it HVS--spectrum}.

\begin{definition}\label{def:spectrum} (\expandafter{\cite{Nem-semi} or \cite[(2.3.1)-(2.3.3)]{BN}}) \
Consider the  H--numbers $\{q^k_\lambda\}_{|\lambda|<1}$ and $\{p^k_\lambda(\pm 1)\}_{\lambda\in S^1}$ of $\mv$.
The {\it extended spectrum} $ESp$ is the union $ESp=Sp\cup ISp$, where
\begin{itemize}
\item[(a)] $Sp$, the {\it spectrum},  is a finite set of real numbers from the interval $(0,2]$ such that any real
number $\alpha$ occurs in $Sp$ precisely $s(\alpha)$ times, where
\[s(\alpha)=\sum_{n=1}^\infty\sum_{u=\pm1}\left(\frac{2n-1-u(-1)^{\lfloor \alpha\rfloor}}
{2}p^{2n-1}_\lambda(u)+np^{2n}_\lambda(u)\right), \ \
(e^{2\pi i\alpha}=\lambda).\]
\item[(b)] $ISp$ is the set of complex numbers from $(0,2]\times i\mathbb{R}$,
$ISp\cap\mathbb{R}=\emptyset$, where $z=\alpha+i\beta$ occurs
in $ISp$ presizely $s(z)$ times, where
\[
s(z)=\begin{cases}
\sum k\cdot q^k_\lambda&\text{if $\alpha\le 1$, $\beta>0$ and $e^{2\pi i z}=\lambda$}\\
\sum k\cdot q^k_\lambda&\text{if $\alpha>1$, $\beta<0$ and $e^{2\pi i z}=1/\bar{\lambda}$}\\
0&\text{if $\alpha\leq 1$ and $\beta<0$, or $\alpha> 1$ and $\beta>0$}.
\end{cases}
\]
\end{itemize}

\end{definition}

Since the size  of a matrices corresponding to $\mv^{2k}_\lambda$
is $2k$ and to $\mw^k_\lambda(\pm 1)$ is $k$, one gets
\begin{equation}\label{cor:es}
|ESp|=\dim U=\deg \det(h-t\Id).
\end{equation}

\subsection{The HVS and spectrum of a link}\label{ss:var}
The {\it variation structure} and {\it H--numbers}
of a link in $S^3$ were defined in \cite{BN}.
Let us review shortly how the construction is performed.

Let $S$ be a Seifert matrix of a link $L$. (For the convention of its definition see
\ref{ss:Lfund}.)  By Keef's result \cite{Keef} $S$ is S--equivalent either to an empty matrix, or
to a matrix $S'$, which can be decomposed into a direct sum
\begin{equation}\label{eq:decomp}
S'=S_0\oplus S_\nd,
\end{equation}
where $S_0$ is a zero matrix and $S_\nd$ is non--degenerate, that is $\det S_\nd\neq 0$. Moreover,
any two such non--degenerate models $S_\nd$ of the same link, are congruent over $\Q$.
The size of $S_0$ is also determined by $L$ (it is equal to $\dim(\ker S\cap \ker S^T)$), we will call it
the {\it irregularity of $L$}, and we will denote it by
\begin{equation}\label{eq:decomp-s0}
\Irr=\Irr(L):=size(S_0).
\end{equation}
Let $n$ be the size of $S_\nd$.
The quadruple $\mv=(U,b,h,V)$, where $U=\mathbb{C}^n$, $V=(S^T_\nd)^{-1}$, $h=VS$, $b=S-S^T$,
constitutes a HVS
 with the sign choice $\varepsilon=-1$. (Here ${\cdot}^T$ denote transposition.)
As changing a Seifert matrix results in congruency of $S_\nd$, which leads to an isomorphism
of variations structures, the structure $\mv$ does not depend on the
choice of a Seifert matrix, so it is a well-defined
link invariant, called $\mv_L$.
Additionally, $\mv_L$ is \emph{simple}. 
Note that $\mv_L$ is defined over the rational numbers ${\mathbb Q}$.
The characteristic polynomial $\cp=\det(h-t\Id)$ of $h$ will be called the \emph{characteristic polynomial of the link}.
Its connection with  Alexander polynomial is as follows (see e.g. \cite[\S 4]{BN}):

\begin{lemma}\label{rem:alex}
Let $\mathcal{V}_L$ be as above.
If the Alexander polynomial $\Delta$ of $L$ is non--zero then
$\Delta=\cp$ up to multiplication
by an invertible element of $\mathbb{Q}[t,t^{-1}]$.
If the Alexander polynomial is zero, then $\cp$ is proportional
to the first higher Alexander polynomial $\Delta_k$, which is not identically zero:
$\Delta_k=0$ for $0\leq k<\Irr$
and $\Delta_{\Irr}=\cp$ (up to an invertible element).
\end{lemma}

\begin{definition} Consider
the integers  $\{q^k_\lambda\}_{|\lambda|<1}$ and $\{p^k_\lambda(\pm 1)\}_{\lambda\in S^1}$
provided by the direct sum decomposition (\ref{eq:mvlsum}) of $\mv_L$. They are called
the \emph{H--numbers} of the link $L$. The associated (extended) spectrum is called the
(extended) spectrum of the link.
\end{definition}

From (\ref{cor:es}) one has  $|ESp|=\deg \cp$.
Moreover, $Sp\setminus \Z$ is symmetric with respect to 1.

\subsection{Mixed Hodge structures and their spectra}\label{ss:MHS-def}

The name and definition of spectrum in Definition \ref{def:spectrum}
 is motivated by the fact that if $L$ is an
{\it algebraic link}, i.e. the link of (local) isolated
plane curve singularity, then $ISp$ is empty and $Sp$  is the classical spectrum
associated with the mixed Hodge structure of the vanishing cohomology (for this  see e.g.
 \cite{Stee,St,Var}).

More generally, let  $f\colon(\mathbb{C}^{n+1},0)\to(\mathbb{C},0)$
be the germ of  an analytic function with isolated singularity at $0$, and
 let $Y$ be the Milnor fiber and  $U=\widetilde{H}_n(Y,\mathbb{R})$.
(For details regarding the Milnor fibration, see e.g. \cite{Milnor-book,AGV,Nem2}.)
One takes the monodromy operator  $h\colon U\to U$, the intersection
form $b\colon U\times U\to\mathbb{R}$ and the
variation operator $V\colon U^*\to U$. One checks (see e.g. \cite{AGV} or \cite[\S\,5]{Nem2}) that
the complexification of  $(U;b,h,V)$ constitutes a $(-1)^n$--HVS.
If $S$ is the Seifert matrix of the Milnor fibration, then at the level of matrices
$V=(S^T)^{-1}$. Since $S$ is unimodular, $V$ will be isomorphism too, hence the
variation structure is simple. For plane curves $\varepsilon=-1$, hence
$h=(S^T)^{-1}S$ and $b=S-S^T$.
The structure $(U;b,h,V)\otimes {\mathbb C}$ is called the `homological HVS' of the germ.

There is dual a  HVS, the `cohomological HVS' associated with the germ,
 which sits on $H^*:=\widetilde{H}^n(Y,\mathbb{C})$.
 Additionally, $\widetilde{H}^n(Y,\mathbb{C})$ carries a limit mixed Hodge structure
 with Hodge filtration $F$ and weight filtration $W$ such that the semisimple part $h_{ss}^*$ of the
  cohomological monodromy operator acts on $(H^*,F,W)$.
They define spectral pairs. In order to eliminate any confusion about
the existing  different normalizations, we
provide some details.

One  considers the
generalized $\lambda$--eigenspaces   $U^*_\lambda$  for all the eigenvalues $\lambda$
of the Gauss--Manin monodromy operator $h_{GM}=(h^*_{ss})^{-1}$
  and the equivariant  (Gauss--Manin) Hodge numbers
 $h_{\lambda}^{p,q}:=\dim\, Gr^p_F\, Gr^W_{p+q}U^*_\lambda$.

 Then these numbers can be codified in a different way
 in the collection of {\it Hodge spectral pairs } of
$(U^*,F,W;h^*_{ss})$. This is a collection of pairs $(\alpha,w)$ from ${\mathbb R}\times{\mathbb N}$
defined by
\begin{equation}\label{eq:Hodgespp}
Spp_{GM}
(f)=\sum_{(\alpha,w)}\ h^{n+[-\alpha],w+s-n-[-\alpha]}_{\exp(-2\pi i\alpha)}
(\alpha,w) \in {\mathbb N}[{\mathbb R}\times {\mathbb N}],
\end{equation}
where $s=1$ if $\lambda=\exp(-2\pi i\alpha)=1$ and $s=0$ otherwise.

This can be transformed in several ways. If by some geometric reason, one wishes to emphasize more
the cohomological monodromy operator $h^*_{ss}$ (instead of $h_{GM}$), one considers
\begin{equation}\label{eq:Hodgespp2}
Spp_*
(f)=\sum_{(\alpha,w)}\ h^{n+[-\alpha],w+s-n-[-\alpha]}_{\exp(2\pi i\alpha)}
(\alpha,w) \in {\mathbb N}[{\mathbb R}\times {\mathbb N}],
\end{equation}

If we forget the weight filtration, then from the
equivariant Hodge filtration one can read the {\it Hodge spectrum}, namely
\begin{equation}\label{eq:Hodgesp}
Sp_*
(f)=\sum\alpha \in {\mathbb N}[{\mathbb R}] \ \ (\mbox{the sum over the spectral pairs $(\alpha,w)$
of $Spp_*(f)$)}.
\end{equation}
Any spectral number $\alpha$ is in the interval $(-1,n)$. Another normalizations of the spectrum identifies it
in the interval $(0,n+1)$:
$Sp_{\MHS}(f)$ is the collection of numbers $(\alpha+1)$ where $\alpha$ runs over the entries of $Sp_*(f)$.

The identification of the Hodge invariants with the associated hermitian variation structure goes through the
crucial polarization property of the mixed Hodge structure. In this way, the cohomological
hermitian variation structure of $f$ can be obtained from $(U^*,F,W)$ by collapsing the Hodge filtration
mod 2, having the collapsed spectral  numbers in $(-1,1]$.
The corresponding H--numbers are, in fact, the equivariant primitive Hodge numbers of  $(U^*,F,W)$
under this collapsing procedure. Usually, the homological and cohomological HVS's do not agree, in the case
$\varepsilon=(-1)^n=-1$ they differ by a sign: $\mv_{coh}=-\mv_{hom}$.
This explains the two slightly different definition of the spectral numbers (\ref{def:spectrum})(a) and (\ref{eq:Hodgespp2}).
Nevertheless, one has the following identification:

\begin{proposition}[\expandafter{\cite[(6.5)]{Nem2}}]\label{prop:mod2}
The HVS--spectrum $Sp_{\HVS}$ is a mod 2 reduction of the Hodge spectrum $Sp_{\MHS}$ considered in
$(0,2]$. In other words
\[Sp_{\HVS}=\{x \bmod 2\colon\, x\in Sp_{\MHS}\}.\]
\end{proposition}

Therefore, for a gem of an isolated plane curve singularity one gets $Sp_{\HVS}=Sp_{\MHS}$.
That means, that the Hodge spectrum can completely be described  from the (real) Seifert form of the link.
This is the model of our further investigation.

\subsection{Spectrum of a link and the Tristram--Levine signature}\label{ss:MHS-TL}
The Tristram--Levine signature (defined first in \cite{Tr,Le})
turn out to be a knot--theoretic counterpart of the spectrum of singular points. We recall
how can they be explicitly expressed from the spectrum of the link.

\begin{definition}\label{def:signature}
Let $L$ be a link and $S$ its Seifert matrix.
The \emph{Tristram--Levine signature} function is the mapping from
$S^1\setminus \{1\}=\{\zeta\in\mathbb{C}\colon|\zeta|=1,\zeta\neq 1\}$ to
$\mathbb{Z}$ given by
\[\sigma_L(\zeta)=\sign \left[(1-\zeta)S+(1-\ol{\zeta})S^T\right].\]
The \emph{nullity} $n_L(\zeta)$ is the nullity
of the same form $(1-\zeta)S+(1-\ol{\zeta})S^T$, while the
 \emph{normalized nullity}, $\nn_L(\zeta)$, is defined as $n_L(\zeta)-\Irr$.
For completeness we extend the definitions  for $\zeta=1$ too. First, we set  $\sigma_L(1)=0$.
Then notice that for any $\zeta\not=1$,
$\nn_L(\zeta)$ equals to the multiplicity of the root of $\cp$ at $\zeta$.
We define $\nn_L(1)$ by this characterization for $\zeta=1$.
\end{definition}

We have the following relation between H--numbers, signatures and  nullities of the link.

\begin{proposition}[\expandafter{\cite[(4.4.6) and (4.4.9)]{BN}}]\label{prop:sig}
Let $Sp=Sp_{\HVS}$ be the real part of the spectrum as defined in
Definition~\ref{def:spectrum}. Let $x\in(0,1)$ and $\zeta=e^{2\pi i x}$. Then
\begin{align*}
\sigma(\zeta)&=-|Sp\cap(x,x+1)|+|Sp\setminus[x,x+1]|+\sum_{n=1}^\infty\sum_{u=\pm 1}up^{2n}_\zeta(u)\\
\nn(\zeta)&=\sum_{k,u}p^k_\zeta(u).
\end{align*}
In particular,
\begin{equation}\label{eq:validonly}
-\sigma(\zeta)+\nn(\zeta)\ge |Sp\cap(x,x+1)|-|Sp\setminus[x,x+1]|.
\end{equation}
\end{proposition}
\begin{remark}
In the cases $x\in\{0,1\}$, the inequality  \eqref{eq:validonly} still holds.
Indeed, the left hand side is non--negative, while the right hand side
is non--positive (since $Sp\setminus \Z$ is symmetric).
Moreover, if 1 is not a root of $\cp$, then \eqref{eq:validonly} is an equality for $x=1$.
\end{remark}
Let us denote
\[D=|Sp\cap\{x,x+1\}|\ge 0.\]
Assume that $\cp$, the characteristic polynomial of the link, has no roots outside the unit circle.
 Then $\deg \cp=|Sp|=|Sp\cap(x,x+1)|+|Sp\setminus[x,x+1]|+D$, hence one also has
 \begin{equation}\label{eq:NEW}
\deg\cp-\sigma(\zeta)+\nn(\zeta)=2|Sp\cap (x,x+1)|+\sum_{\substack{k \ odd\\u=\pm 1}}p^k_\zeta(u)
+\sum_{k\ even}2p^k_\zeta(-1)+D.
\end{equation}

For any $x\in [0,1]$, parallel to the set $Sp\cap (x,x+1)$, we will also consider the set
$Sp\setminus [x,x+1]=Sp\cap(0,x)+Sp\cap(1+x,2].$
These two types cover all the `length one open intervals' of the mod 2  spectrum.

\vspace{2mm}

The following corollary will be used extensively in the sequel.

\begin{corollary}\label{cor:univ}
Let $L$ be a link and $\cp$ its characteristic polynomial.
Assume that $\cp$
has no roots outside the unit circle.
If $\zeta=e^{2\pi ix}$ is not a root of $\cp$  then
\[|Sp\cap(x,x+1)|=\frac{1}{2}\left(\deg \cp-\sigma(\zeta)\right)
\ \ \mbox{and} \ \
|Sp\setminus[x,x+1]|=\frac12\left(\deg\cp+\sigma(\zeta)\right).\]
Moreover, for arbitrary $x\in[0,1]$:
\begin{equation}\label{eq:boundspec}\begin{split}
\frac{1}{2}\left(\deg\cp-\sigma(\zeta)+\nn(\zeta)\right)\ge |Sp\cap(x,x+1)|\\
\frac{1}{2}\left(\deg\cp+\sigma(\zeta)+\nn(\zeta)\right)\ge |Sp\setminus [x,x+1]|.
\end{split}\end{equation}
\end{corollary}

\subsection{Morse theory of plane curves}\label{S:morseplane}
For any $\xi\in{\mathbb C}^2$ and $r>0$ let $B(\xi,r)$ be the ball centered at $\xi$ and with radius $r$,
also $S^3(\xi,r):=\partial B(\xi,r)$.
For an algebraic curve $C$ sitting in $\CC^2$,
we write $\normal{(C\cap B(\xi,r))}$ for the normalization of $C\cap B(\xi,r)$, and the genus
of $C\cap B(\xi,r)$ is the genus of its normalization.

For any  link $L$, we denote by  $c_L$ its number of components, and we set
\begin{align*}
w_L(\zeta)&:=-\sigma_L(\zeta)+1-c_L+n_L(\zeta)\\
-u_L(\zeta)&:=\sigma_L(\zeta)+1-c_L+n_L(\zeta).
\end{align*}

\begin{remark}
The convention used in  \cite{Bo} is that $n_L$ is the dimension of the kernel of
$(1-\zeta)S+(1-\ol{\zeta})S^T$ \emph{increased by $1$},
this explains the formal differences compared with \cite{Bo}.
\end{remark}

We will also fix $\zeta\in S^1\setminus\{1\}$.
Let us begin by citing a result from \cite{Bo}.
\begin{proposition}[\expandafter{\cite[Proposition~6.8]{Bo}}]\label{prop:morse-cite}
Let $\xi$ be a generic point of $\mathbb{C}^2$ and $r_0<r_1$ two values
such that the intersections $L_i:=C\cap S^3(\xi,r_i)$  are
transverse ($i=0,1$).  With the notations $c_i=c_{L_i}$,
$g_i=$ the genus of $C_i:=C\cap B(\xi,r_i)$ and $k_i=$ the number
of connected components of $\normal{C_i}$,
one has
\begin{equation}\label{eq:wbound}
\begin{split}
w_{L_1}(\zeta)-\sum w_{\LS_k}(\zeta)-w_{L_0}(\zeta)&\ge -2(g_1-g_0+c_1-c_0-k_1+k_0),\\
-\left(u_{L_1}(\zeta)-\sum u_{\LS_k}(\zeta)-u_{L_0}(\zeta)\right)&\ge -2(g_1-g_0+c_1-c_0-k_1+k_0),
\end{split}
\end{equation}
where $\LS_k$ are the links of singularities of $C$, which lie in $B(\xi,r_1)\setminus B(\xi,r_0)$.
\end{proposition}

We use Proposition~\ref{prop:morse-cite} in two special cases.
\begin{corollary}\label{cor:smooth-morse}
Let $C_0$ and $C_1$ be as in \ref{prop:morse-cite}.
If $C_{01}=C_1\setminus C_0$ is smooth then
\begin{align*}
-\sigma_{L_1}(\zeta)+n_{L_1}(\zeta)-(-\sigma_{L_0}(\zeta)+n_{L_0}(\zeta))&\ge \chi(C_{01}).\\
\sigma_{L_1}(\zeta)+n_{L_1}(\zeta)-(\sigma_{L_0}(\zeta)+n_{L_0}(\zeta))&\ge \chi(C_{01}).
\end{align*}
\end{corollary}
\begin{proof}
Use the definition of $w$, \eqref{eq:wbound} and  $\normal{C_{01}}=C_{01}$ for the first inequality. For the second one
we use $-u$ instead of $w$.
\end{proof}

The other important application is if $r_0$ is small, so that $L_0$ is an unknot.

\begin{proposition}
Fix  $r$ such that the intersection $C\cap S(\xi,r)$ is transverse, and set $L:=C\cap S(\xi,r)$.
Let $C_{smooth}$ be the smoothing of $C\cap B(\xi,r)$ (e.g. if $C$ is given
by $F^{-1}(0)$ for some reduced polynomial, then $C_{smooth}$ can be taken
as $F^{-1}(\varepsilon)\cap B(\xi,r)$ for $\varepsilon$ sufficiently small). Let $z_1,\dots,z_k$
be the singular points of $C\cap B(\xi,r)$ with links    $\LS_1,\dots,\LS_k$,
Milnor numbers $\mu_1,\dots,\mu_k$, number of branches
$c_1,\dots,c_k$, and signatures $\sigma_1(\zeta),\dots,\sigma_k(\zeta)$. Then
\begin{equation}\label{eq:diff-r0}
\begin{split}
-\sigma_{L}(\zeta)+n_{L}(\zeta)+(1-\chi(C_{smooth}))&\ge \sum_{j=1}^{k}\left(-\sigma_{\LS_j}(\zeta)
+n_j(\zeta)+\mu_j\right)\\
\sigma_{L}(\zeta)+n_{L}(\zeta)+(1-\chi(C_{smooth}))&\ge \sum_{j=1}^{k}\left(\sigma_{\LS_j}(\zeta)
+n_j(\zeta)+\mu_j\right).
\end{split}
\end{equation}
\end{proposition}
\begin{proof}
We prove only the first part, in the second one we use $-u_L$ instead of $w_L$.

Let $r_{min}$ be  minimal with $C\cap S(\xi,r)$ non-empty, and set
$r_0:=r_{min}+\varepsilon$ for $\varepsilon$ sufficiently small.
Then $L_0$ is an unknot with $w_{L_0}(\zeta)\equiv 0$, $c_0=k_0=1$, thus
\eqref{eq:wbound} gives
\begin{multline*}
-\sigma_{L}(\zeta)+n_{L}(\zeta)+1-c_L\ge
\\
\ge \sum_{j=1}^{k}(-\sigma_j(\zeta)+n_j(\zeta)+\mu_j)-\sum_{j=1}^{k}(\mu_j+c_j-1)-2g(C)-2c_L+2k_1.
\end{multline*}
The proof is completed by applying the genus formula
$2(g(C_{smooth})-g(C))=\sum_{j=1}^k(\mu_j+c_j-1)$, the fact that $b_1(C_{smooth})=2g(C_{smooth})+c_L-1$ and observing that $b_0(C_{smooth})\le 2k_1$ (it is even
bounded by $k_1$ alone).
\end{proof}

\begin{remark}
The cited result (i.e. Proposition~\ref{prop:morse-cite}) does not really require
Morse theoretical arguments, although they are very convenient.
We could deduce it --- with approximately the same amount of work
--- from the Murasugi inequality \cite[Theorem~12.3.1]{Kaw-book} too.
The argument is that $C_{01}=C\cap(B(\xi,r_1)\setminus B(\xi,r_0))$ induces
a cobordism between the links $L_0':=L_0\amalg\LS_1\amalg\ldots\amalg\LS_j$ and $L_1$. In this way we
do not use anywhere that $C$ is a complex curve, only that its genus is the difference of the genera of the
minimal Seifert surfaces of $L_1$ and $L_0'$.
\end{remark}

\section{The Seifert form and the MHS  of a polynomial at infinity}\label{s:HVS-full}
In this section we compare the Hodge--spectrum associated with the limit mixed Hodge structure of
a polynomial map  at infinity with the HVS--spectrum provided by its regular link at infinity.
In this way {\it we recover the Hodge--spectrum from the `Seifert form at infinity'}.
For results concerning the limit mixed Hodge structure (MHS)
 and the Hodge spectrum at infinity the reader might consult
\cite{SSS,Di,Br}.

\subsection{Basic definitions}\label{S:basdef2}
Let $F\colon\mathbb{C}^2\to\mathbb{C}$ be a reduced polynomial with critical values
 $x_1,\dots,x_N$. Since
$\mathbb{C}^2$ is not compact, the topology of a fiber $F^{-1}(y)$ can vary
even if $y$ changes in a set of regular values of $F$.

\begin{definition}(\cite{Neu1})
The fiber $F^{-1}(c)$ is called \emph{regular at infinity} if
there exists a (small) disk $D\ni c$ in $\mathbb{C}$ and a (large) ball
$B\subset\mathbb{C}^2$ such that $F$ restricted to
$F^{-1}(D)\setminus B$ is a trivial fibration.
The fiber is called {\it irregular at infinity}
if it is not regular at infinity.
\end{definition}
Consider all the values $y_1,\dots,y_M$ such that $F^{-1}(y_k)$
is not regular at infinity. Set
$\rho\in\mathbb{R}$ with $\rho>\max_{k,l}\{ |x_k|,|y_l|\}$  and set
$\gamma=\{z\in\mathbb{C}\colon |z|=\rho\}$.
Then $F$ restricted to $F^{-1}(\gamma)$ is a locally trivial fibration, called the
{\it fibration of $F$ at infinity}. It will be denoted  $\fbi$.
The fiber of $\fbi$ is the (generic) fiber $Y_\infty:=F^{-1}(\rho)$ of $F$.
The induced algebraic monodromy over $\gamma$, called
the \emph{monodromy  of $F$ at infinity}, will be denoted by
 \begin{equation}\label{eq:moninf}
 h_\infty\colon H_1(Y_\infty,\mathbb{Z})\to H_1(Y_\infty,\mathbb{Z}).\end{equation}
Furthermore, we consider on $H_1(Y_\infty,\mathbb{Z})$ the intersection form $b_\infty$ too.
Already this {\it isometric structure} $(H_1(Y_\infty);b_\infty,h_\infty)$
 contains important information about the behavior of $F$ at infinity, nevertheless,
we will enhance it  in two different
ways. The first  is topological: we investigate the possibility to extend the pair $(b_\infty,h_\infty)$ to
a variation structure (this, strictly speaking, in general, only `partially' is
 possible). The candidate for the variation operator
is the inverse transpose of the Seifert matrix of the link at infinity.
The second is algebraic: one lifts the pair  $(b_\infty,h_\infty)$ to the level of a polarized mixed Hodge structure by
considering the limit mixed Hodge structure of $F$ at infinity.

\vspace{2mm}

First we start with the topological part.

Fix a fiber $F^{-1}(c)$ which is regular at infinity. For sufficiently large $R$ the intersection $F^{-1}(c)$
with $\partial B(R)$ is transverse. This link $F^{-1}(c)\cap\partial B(R)\subset \partial B(R)$, denoted by $\reg{L}$, is
independent (up to isotopy) of  $R$ and $c$.
It  is called the \emph{regular link at infinity} of $F$.

According to \cite[Theorem~5]{Neu1} we can associate with $\reg{L}$ the so--called
fundamental {\it multilink}  at infinity $L_{fund}$, which is fibered. This means the following:
there exist a link $L_{fund}$ with components $\{L_{fund,i}\}_{i=1}^\nu$ and positive
multiplicities ${\bf n}=\{n_i\}_{i=1}^\nu$ such that there is a fibration $\phi:S^3\setminus L_{fund}
\to S^1$ with the following property:
for any closed loop  $\tau\in S^3\setminus L_{fund}$,
$\phi_*([\tau])\in H_1(S^1)={\mathbb Z}$ equals the linking number of $[\tau]$ with
$\sum_in_iL_{fund,i}$. Furthermore, the closure $\overline{Y_t}$ of
the fiber $Y_t=\phi^{-1}(e^{2\pi it})$ ($t\in [0,1]$)
is not a manifold with boundary, but  homologically $\overline{Y_t}\setminus Y_t$
is the multilink $\sum_in_iL_{fund,i}$.

Finally, the connection between the multilink $L_{fund}$ and the link at infinity $\reg{L}$
is the following. Let $T=T(L_{fund})$ be a closed small tubular neighbourhood of $L_{fund}$.
Then $\reg{L}$  is the intersection of a fiber $Y_0$ with $\partial\, T(L_{fund})$.

\begin{lemma}\label{lemma:n_i}
For any  $i\in\{1,\ldots,\nu\}$, let $l_i$ be the linking number
\[l_i=\lk(L_{fund,i},\sum_{j\not=i}n_jL_{fund,j})\]
and $n_i'$ the (positive) greatest common divisor
of $n_i$ and $l_i$.   Then the number of components of
$Y_0\cap \partial\, T(L_{fund,i})$
is exactly $n_i'$. Hence, $\reg{L}$ has $\sum_{i=1}^\nu n_i'$ components.
Moreover, the components of $Y_0\cap \partial\, T(L_{fund,i})$ are cyclically permuted by the
geometric monodromy of $\phi$.
\end{lemma}
\begin{proof}
See \cite[\S\,3 and 4]{EN}.
\end{proof}

Another important point about $L_{fund}$ is that its fiber $Y_0$ can be identified with
the generic fiber $Y_\infty$ of the polynomial $F$ \cite[Theorem 4]{Neu1}. In fact, by
\cite[Theorem~1.1]{AC} one has

\begin{lemma}\label{prop:hvs1}
The multilink fibration $S^3\setminus L_{fund}\to S^1$
associated with $(L_{fund},{\bf n})$ and the fibration $Fib_\infty$ of $F$ are isomorphic.
\end{lemma}

By \cite[page 37]{EN}, $Y_0$ has $d=gcd_i\{n_i\}$ connected components.
In the sequel we will assume that $d=1$, that is,
 {\it the generic fiber of $F$  is connected}.


\subsection{The multilink Seifert form of $L_{fund}$.}\label{ss:Lfund}
The surface $Y_0$, the fiber of the multilink $(L_{fund},{\bf n})$, is a generalized Seifert surface of the multilink, cf. \cite[pages 28-29]{EN}.
In the sequel we refer
to it as the {\it multilink Seifert surface}.  Using this surface, one can define {\it multilink
Seifert form} associated with $Y_0$, cf. \cite[\S 15]{EN}. It is a bilinear form on
$H_1(Y_0,\mathbb{Z})$ defined similarly as the classical Seifert form, namely
 $S_{fund}(\alpha,\beta)$ for
$\alpha,\beta\in H_1(Y_0,\mathbb{Z})$ is the linking number $\lk(\alpha,\beta^+)$, where $\beta^+$ is the push-forward of $\beta$ in the positive direction.

If all the multiplicities $\{n_i\}_i$ equal 1, then $L_{fund} $  is a fibred link, and $S_{fund}$ is
its classical Seifert form, hence it has determinant $\pm 1$. In the case of general multiplicity
system ${\bf n}$ this is not the case anymore. In fact, $S_{fund}$ can be even degenerate.
Nevertheless, some parts of the classical theory survive.

\begin{lemma}\label{prop:Se} Let $H^*$ denote the dual of $H$, $T^o$ the interior of $T$, and $\overline{Y}_{[a,b]}:=
\bigcup_{a\leq t\leq b}\overline{Y_t}$.

(a) The groups $H_1(\overline{Y_0},\mathbb{Z})$ and $H_1(Y_0,\mathbb{Z})^*$
are isomorphic. In fact one has the following sequence of isomorphisms, denoted by $s$:
$$H_1(\overline{Y_0})
\stackrel{\partial^{-1}}{\longrightarrow}
H_2(S^3,\overline{Y_0})
\stackrel{(1)}{\longrightarrow}
H_2(S^3,\overline{Y}_{[0,\frac{1}{2}]})
\stackrel{(2)}{\longrightarrow}
H_2(\overline{Y}_{[\frac{1}{2},1]},\ \overline{Y_{1/2}}\cup \overline{Y_1} )
\stackrel{(3)}{\longrightarrow}
$$
$$
H_2(\overline{Y}_{[\frac{1}{2},1]},\ \overline{Y_{1/2}}\cup \overline{Y_1} \cup
(T\cap \overline{Y}_{[\frac{1}{2},1]}))
\stackrel{(4)}{\longrightarrow}
H_2(\overline{Y}_{[\frac{1}{2},1]}\setminus T^o,\
\partial(\overline{Y}_{[\frac{1}{2},1]}\setminus T^o))
$$ $$\stackrel{(5)}{\longrightarrow}
H_1(\overline{Y}_{[\frac{1}{2},1]}\setminus T^o)^*
\stackrel{(6)}{\longrightarrow}
H_1(\overline{Y_1}\setminus T^o)^*
\stackrel{(7)}{\longrightarrow}
H_1({Y_1})^*=H_1({Y_0})^*.
$$

(b) Let $j:H_1(Y_0,\mathbb{Z})\to H_1(\overline{Y_0},\mathbb{Z})$ be induced by the inclusion.
Then the composition $$H_1(Y_0,\mathbb{Z})\stackrel{j}{\longrightarrow}
 H_1(\overline{Y_0},\mathbb{Z})\stackrel{s}{\longrightarrow}H_1(Y_0,\mathbb{Z})^*$$
 can be identified with the multilink Seifert form $S_{fund}$.

(c) Identify the isometric structure $(b_\infty,h_\infty)$ with
the intersection form and monodromy of $H_1(Y_0)$ (by \ref{prop:hvs1}). Then, in matrix notation,
$$b_\infty=S_{fund}-S_{fund}^T, \ \  \ \mbox{and } \ \ \ S_{fund}^Th_\infty=S_{fund}.$$
In particular, $h_\infty$ is an automorphism of $S_{fund}$, that is $h_\infty^TS_{fund}h_\infty=S_{fund}$.
\end{lemma}
\begin{proof}
In the sequence of isomorphisms $\partial^{-1}$ comes from the exact sequence of the pair;
(1), (3), (6) and (7) are induced by deformation retracts;
(2) and (4) are excisions, while (5) is provided by duality of the manifold with boundary
$\overline{Y}_{[\frac{1}{2},1]}\setminus T^o$. Part (b) and (c) follow by similar
argument as in the classical case, see e.g. the survey \cite[(3.15)]{Nem-Eger}.
\end{proof}
In the above composition, although $s$ is an isomorphism, $j$ in general is not. Since,
by our assumption $\tilde{H}_0(Y_0,\mathbb{Z})=0$, $j$ can be inserted in the following
long exact sequence:
\begin{equation}\label{eq:FF}
0\to H_2(\overline{Y_0})\to H_2(\overline{Y_0}, Y_0)\to H_1(Y_0)\stackrel{j}{\longrightarrow}
H_1(\overline{Y_0})\to H_1(\overline{Y_0}, Y_0)\to 0.
\end{equation}
\begin{lemma}\label{prop:Sfund}
\begin{equation*}\begin{array}{rll}
(a) \ \ \ \ & H_2(\overline{Y_0}, Y_0,\mathbb{Z})& =\oplus_{i=1}^\nu \mathbb{Z}^{n_i'-1},\\
(b) \ \ \ \ & H_1(\overline{Y_0}, Y_0,\mathbb{Z})& =\oplus_{i=1}^\nu (\, \mathbb{Z}^{n_i'-1}\oplus \mathbb{Z}_{n_i/n_i'}).\end{array}\end{equation*}
In particular, $ H_2(\overline{Y_0},\mathbb{Z}) =0$ and
\begin{equation}\label{eq:Ndef}
\dim\ker j=\sum_{i=1}^\nu(n'_i-1).
\end{equation}
\end{lemma}
\begin{proof}
By excision and deformation retract argument $H_q(\overline{Y_0}, Y_0)=\oplus_i H_q(A_i,B_i)$,
where $$(A_i,B_i):=(\overline{Y_0}\cap T(L_{fund,i}),\  Y_0\cap \partial
 T(L_{fund,i})).$$
Note that the homotopy type of $A_i$ is $L_{fund,i}$, while of
$B_i$ is $n_i'$  copies of $S^1$. Each of these copies maps (via the inclusion $B_i\hookrightarrow
  A_i$) onto $L_{fund,i}$ as the $n_i/n_i'$--covering. Therefore, the inclusion
   $B_i\hookrightarrow A_i$ at $H_1$--level is $\Z^{n_i'}\to \Z$, $\{a_1,\dots,a_{n'_i}\}\mapsto \frac{n_i}{n_i'}\cdot\sum_ka_k$. This gives  (a) and (b).
The rest follow by rank computation argument from (\ref{eq:FF}).
\end{proof}

Next, we will consider another compactification $\tY$ of $Y_0$. Denote $Y_0^o:=Y_0\setminus T(L_{fund})^o$, the complement
of the interior of the tube. $\partial Y_0^o$ consists of $\sum_in_i'$ copies of $S^1$. Let $\tY$ be obtained from $Y_0^o$
by gluing to each boundary circle a 2--disc, in this way obtaining a compact smooth surface. In fact, the fibration
at infinity $F$ over $\gamma$ can be compactified (even algebraically) to a fibration $\widetilde{F}$
over $\gamma$ with smooth compact fibers $\tY$, where in this language the compact fibers consists of $Y_0$ with
additionally $\sum_in_i'$ `points at infinity'. This point of view is used in Hodge theoretical computations, see e.g.
\cite{Di0} or \cite[\S 3]{Di}.

One has the following exact sequence:
\begin{equation}\label{eq:exact2}
0\to\Z\to \oplus_i\Z^{n_i'}\to H_1(Y_0,\Z)\to H_1(\tY,\Z)\to 0.
\end{equation}
Above, $\oplus_i\Z^{n_i'}$ is generated by the discs, their images in $H_1(Y_0)$ are the classes of the
circles $\partial Y_0^o$. $\Z$ from the left is  $H_2(\tY)$; its image is generated by $\partial Y_0^o$.
The monodromy extends to $H_1(\tY)$ (or to $\widetilde{F}$) (and will be denoted by $\widetilde{h}_\infty$),
and also to the discs/points at infinity:
it acts trivially on $\Z$, on $\Z^{n_i'}$ acts by permutation of the base elements (denoted by $h_{per}$).

Let $\widetilde{b}_\infty$ be the intersection form on $H_1(\tY)$.

\begin{lemma}\label{lem:HODGE} \

(a) $\widetilde{h}_\infty$ has no eigenvalue 1, and all its Jordan blocks have size not larger than two.

(b) The exact sequence (\ref{eq:exact2}), together with the algebraic monodromy action on it, splits. That is,
$h_\infty$ has no Jordan block of size three,  and the blocks of size two
of $h_\infty$ and $\widetilde{h}_\infty$ agree.
In other words, over $\Q$, one has a direct sum decomposition:
\begin{equation}\label{eq:decHODGE}
(H_1(Y_0);b_\infty,h_\infty)=(H_1(\tY);\widetilde{b}_\infty,\widetilde{h}_\infty)\oplus
(\oplus_i\Q^{n_i'}/\Q;0,h_{per}).
\end{equation}
Moreover, $(\widetilde{b}_\infty,\widetilde{h}_\infty)$ is a non--degenerate isometric structure.

(c) All roots of $h_{\infty}$ are roots of unity.
\end{lemma}
\begin{proof}
The statements follow from  the mixed Hodge theory of the degeneration at infinity of $F$ and $\tilde{F}$.
Part (a) is proved e.g. in \cite{Di0,Di}. Part (b) follows from the spectral pair computation of
the mixed Hodge structure carried on $H^1(Y_0,{\mathbb C})$. More precisely, there is a cohomological analogue
of the sequence (\ref{eq:exact2}) which carries mixed Hodge structure compatible with
the action of the monodromy, see again \cite[\S\,3]{Di}.
The number of Jordan blocks of size two correspond to those  spectral pairs $(\alpha,w)$ for which $w=0$.
These are computed for both $H^1(Y_0,{\mathbb C})$ and $H^1(\tY,{\mathbb C})$ in \cite{Br}, and their number agree.
For (c) use e.g. the Monodromy Theorem for $\widetilde{F}$ at infinity.
\end{proof}

Finally, we summarize the properties of $L_{fund}$ in the following proposition.
As usual, if $h$ is an automorphism of the vector space $V$, then $V_{\lambda=1}$ denotes the generalized
eigenspace corresponding to eigenvalue 1, while $V_{\lambda\not=1}$ is the direct sum of the other generalized
eigenspaces.

\begin{proposition}\label{prop:SS}
Set $U:=H_1(Y_0,\mathbb{Q})$ and let $b_\infty$ and $h_\infty$ be the intersection form and the algebraic monodromy induced by the multilink fibration
$\phi:S^3\setminus L_{fund}\to S^1$.

Then the following facts hold:

(a) $Y_0$ is the minimal multilink Seifert surface of the multilink $(L_{fund},{\bf n})$, and all
minimal multilink Seifert surfaces of $(L_{fund},{\bf n})$ are isotopic to $Y_0$.

(b) One  has a direct sum decomposition (Keef decomposition, cf. (\ref{eq:decomp})):
\begin{equation}\label{eq:decfund}
(U,S_{fund})= (U_0\oplus U_{ndeg}, S_{fund,0}\oplus S_{fund,ndeg})\end{equation}
such that $S_{fund,0}=0$ of size
$\Irr=\sum_{i=1}^\nu \ (n_i'-1)$, and $S_{fund,ndeg}$ is non--degenerate.

A possible free generator set for $U_{0}$ is the collection of the cycles
$ L^{\infty}_{reg,i,k}-L^{\infty}_{reg,i,k+1}$ \
$(1\leq i\leq \nu;\ 1\leq k< n_i')$,
where
$\{L^{\infty}_{reg,i,k}\}_{k=1}^{n_i'}$ are the components of  $Y_0\cap \partial T(L_{fund,i})$.

(c) The compatibility of the decompositions (\ref{eq:decfund}) and  (\ref{eq:decHODGE})
is the following:

\vspace{1mm}

(c.1) \ $(U_{ndeg})_{\lambda\not=1}=H_1(\tY)$. On this space $(S_{fund,ndeg})_{\lambda\not=1}$
completes the non--degenerate isometric structure $(\widetilde{b}_\infty,\widetilde{h}_\infty)$
to a simple $(-1)$--variation structure.

\vspace{1mm}

(c.2) \  $(\oplus_i\Q^{n_i'}/\Q;0,[h_{per}])=((U_{ndeg})_{\lambda=1};0,\Id)\ \oplus\ (U_0;0,h_{\infty}|_{U_0})$
(and this is an eigenspace decomposition).

\vspace{1mm}

(c.3) \ $(U_{ndeg})_{\lambda=1}$ has dimension $\nu-1$,
on it the restriction of $b_\infty$ is trivial, the restriction of $h_\infty$
is the identity, and this degenerate isometric structure is completed by $(S_{fund,ndeg})_{\lambda=1}$
to a simple variation structure.

\vspace{1mm}

(c.4) \ On $U_0$ the restrictions of $b_\infty$ and $S_{fund}$ are trivial
(hence all the equivariant signature type invariants
including the Tristram--Levine signatures of restrictions of
$S_{fund}$ and $(b_{\infty},h_{\infty})$ are the same).
Nevertheless, the restriction of $h_{\infty}$ is non--trivial (in fact, it has no eigenvalue 1),
 hence the isometric structure cannot be
completed to a variation structure. The characteristic polynomial of the restriction of $h_\infty$ is
$$\det(h_{\infty}|_{U_0}-t\,\Id )=\prod_i\frac{t^{n_i'}-1}{t-1}.$$
\end{proposition}
\begin{proof}
(a) follows from \cite[(4.1)]{EN}. For (b) note that the generators listed are in the kernel of
$j$. For this use e.g. the proof of (\ref{prop:Sfund}), where
$\{L^{\infty}_{reg,i,k}\}_{k=1}^{n_i'}$ are exactly the components of $B_i$. Another possibility is
check directly that $S_{fund}(L^{\infty}_{reg,i,k}-L^{\infty}_{reg,i,k+1},\beta)=
S_{fund}(\beta,L^{\infty}_{reg,i,k}-L^{\infty}_{reg,i,k+1})=0$ for any $\beta$. Indeed,
if $T$ is sufficiently small tubular neighborhood, then it does not intersect $\beta$, on the other hand
 inside of $T$ the circles  $L^{\infty}_{reg,i,k}$ and $L^{\infty}_{reg,i,l}$ are homologous.
 For part (c) use Lemmas \ref{prop:Se}(c) and \ref{lem:HODGE};
 for the characteristic polynomial use the fact
 that the components $\{L^{\infty}_{reg,i,k}\}_{k}$ are cyclically
 permuted, cf. Lemma  \ref{lemma:n_i}.
\end{proof}

The multilink structure $(L_{fund},S_{fund})$ now will be used in two different aspects. First, it can be
related with the link $\reg{L}$;  in fact,  one can  recover it from $\reg{L}$, see (\ref{ss:SSSS}). On the other hand,
the multilink fibration of $L_{fund}$ can be identified with the fibration  at infinity $Fib_\infty$ of $F$,
cf. Lemma \ref{prop:hvs1}.
In this way $L_{fund} $ creates the bridge between $\reg{L}$ and $Fib_\infty$.

\subsection{The Seifert form of $\reg{L}$.}\label{ss:SSSS} Set $Y_0^o:=Y_0\setminus T(L_{fund})^o$ as above. Obviously,
$Y_0^o\hookrightarrow Y_0$ admits a deformation retract, hence $H_1(Y_0^o,\mathbb{Z})=H_1(Y_0,\mathbb{Z})$
canonically.

\begin{lemma}\label{lem:Lreg} One has the following facts:

(a)  $Y_0^o$ is the minimal Seifert surface of  $\reg{L}$, and all
minimal Seifert surfaces of $\reg{L}$ are isotopic to $Y_0^o$.

(b) The Seifert form $S_{reg}$ of  $\reg{L}$   associated with $Y^o_0$ is identical with $S_{fund}$
(under the identification $H_1(Y_0^o,\mathbb{Z})=H_1(Y_0,\mathbb{Z})$). In particular, all the result
listed in Proposition \ref{prop:SS} about $S_{fund} $  are valid for $S_{reg}$ too.

(c) Let $h_{reg}$ be the monodromy of the variation structure associated with
$S_{reg,ndeg}=S_{fund,ndeg}$ (as in \ref{ss:var}).
Then the higher Alexander polynomials $\Delta_k$ of $\reg{L}$ satisfies the following  identities:
$\Delta_k\equiv 0$ for $0\leq k<\Irr$, $\Delta_{\Irr}(t)=\det(h_{reg}-t\,\Id)$.

(d) All the roots of the (higher)  Alexander polynomial $\Delta_{\Irr}$ of $\reg{L}$
are roots of unity.
\end{lemma}
\begin{proof}
(a)--(b)--(c) follows from Proposition \ref{prop:SS} and from the construction of $Y_0^o$.
For (d) use either Lemma \ref{lem:HODGE}(c) or note that the multilink fibration of $L_{fund}$ can be represented by
a splice diagram \cite{Neu1,Neu2}, hence the characteristic polynomial of $h_{fund}$ is a product of
cyclotomic polynomials by \cite[Theorem~13.6]{EN}.
\end{proof}

\subsection{The HVS--spectrum of the regular link at infinity,  $\reg{L}$.}\label{ss:Spec} \

For local isolated plane curve singularities we have the following
 classical result, which in the language of H--numbers
$p^k_\lambda(\pm 1)$ and $q^k_\lambda$ of their local links
can be formulated as follows (see e.g. \cite[Proposition~3.1.5, Lemma~3.1.6]{BN}).

\begin{proposition}[Monodromy Theorem]\label{prop:monodlocal}
Let $L$ be an algebraic link and $p^k_\lambda(\pm 1)$, $q^k_\lambda$ its H--numbers. Then

\vspace{1mm}

(a) \ $q^k_\lambda=0$ for all $k>0$ and $|\lambda|<1$; moreover $p^k_\lambda(\pm 1)=0$ unless $\lambda$ is a root of unity,

(b) \ $p^k_\lambda(\pm 1)=0$ for all $k>2$. Moreover $p^2_1(\pm 1)=0$,

(c) \  $p^2_\lambda(-1)=0$ and $p^1_1(-1)=0$.
\end{proposition}

The fundamental mulitlink at infinity $L_{fund}$, or the regular link at infinity $\reg{L}$, in general, cannot be realized by a local algebraic link.
However, their H--numbers share similar properties as the H--numbers of local links.

\begin{proposition}\label{prop:HN}
For the H--numbers of $\reg{L}$ the following facts hold.

\vspace{1mm}

(a) \ $q^k_\lambda=0$ for all $k>0$ and $|\lambda|<1$, moreover  $p^k_\lambda(\pm 1)=0$ unless $\lambda$ is a root of unity,

(b) \  $p^k_\lambda(\pm 1)=0$ for all $k>2$, and  $p^2_1(\pm 1)=0$,

(c) \  $p^1_1(-1)=0$,

(d) \ $p^2_\lambda(1)=0$ for $\lambda\not=1$.
\end{proposition}
\begin{proof} (a) and (b)   follow from Lemma \ref{lem:HODGE}. Next we prove (c).
First we recall that ${\mathcal W}^1_1(\pm 1)=({\mathbb C}; 0, \Id, \mp 1)$, hence we have to show that
the restriction of
$S_{fund}$ on $U_{\lambda=1}$ is negative definite. This follows from the more general result
Proposition~\ref{prop:negative} of section \ref{sec:negative}.

(d)
By \cite{Neu1,NeRu} $L_{fund}$ can is represented by a splice diagram with all edges have negative
determinants. Thus, $L_{fund}$
has uniform twists (all positive) (see \cite[Chapter 14]{EN}). Therefore, by the discussion
in \cite[Section 2]{Neu-inv} we have for each $x\in U_\lambda$
\[S_{fund}(\lambda x,(h_\infty-\lambda)x)\ge 0.\]
This shows that $p^2_\lambda(+1)$ cannot occur.
\end{proof}

Let $Sp_{\HVS}(\reg{L})$ be the HVS--spectrum associated with the link $\reg{L}$.

\begin{corollary}\label{lem:dis} \
(a) The HVS--spectrum of $U_{\lambda=1}$ consists of  $(\nu-1)$ copies of $(1)$.

(b) All elements of $Sp_{\HVS}(\reg{L})$ are situated in $(0,2)$, and $Sp_{\HVS}(\reg{L})$
is symmetric with respect to 1.
\end{corollary}
\begin{remark}
The proofs of Propositions~\ref{prop:monodlocal} and \ref{prop:HN} rely on some key properties of the
splice diagrams of the corresponding links. The common properties (which imply the common
$p^1_1(-1)=0$) is that in both cases the `multiplicities of the nodes' and the `(near) weights' are positive.
The crucial difference between the diagrams is that in the local case the edge determinants are positive, while
for the diagram at infinity they are negative.
This implies the sign difference in the $p^2_\lambda(\pm 1)$--vanishing.
For more details, see subsection \ref{sec:negative}.
\end{remark}

The next identity  will be often used in the sequel.
\begin{corollary}\label{prop:degcp}
If $Y_\infty$ is a regular fiber of $F$, then
$1-\chi(Y_\infty)=\deg\Delta_{\Irr}(\reg{L})+\Irr.$
\end{corollary}
\begin{proof}
By Lemma~\ref{lem:Lreg}(a) the size of $S_{reg}$ is equal to $1-\chi(Y_\infty)$.
On the other hand, $\deg\Delta_{\Irr}$ is equal to the size of $S_{ndeg}$ by Lemma~\ref{lem:Lreg}(c). The difference of the sizes of the two matrices
is equal to $\Irr$ by Proposition~\ref{prop:SS}(b) (cf. also Lemma~\ref{lem:Lreg}(b)).
\end{proof}

\subsection{The Hodge--spectrum of the fibration of $F$ at infinity.}\label{S:MHSinf}
Let $Sp_{\MHS,\infty}$ be the spectrum associated with the limit mixed Hodge structure of $F$ at infinity
defined in a similar way as in subsection \ref{ss:MHS-def}.
The main result of this subsection shows that $Sp_{\MHS,\infty}$ can be recovered from the
rational Seifert form of $\reg{L}$ and from the integers $\{n_i'\}_i$.
Conversely, $Sp_{\HVS}(\reg{L})$ is the maximal subset
of $Sp_{\MHS,\infty}$, which is symmetric with respect to $1$.

More precisely, in $\Z[\Q]$ one has

\begin{theorem}\label{prop:finalspeccompare}
\[Sp_{\MHS,\infty}=Sp_{\HVS}(\reg{L})+ \sum_{i=1}^\nu \ \Big(\frac{1}{n'_i}\Big)+\dots +\Big(\frac{n'_i-1}{n'_i}\Big).\]
\end{theorem}
\begin{proof} Let us consider the decomposition given in Proposition \ref{prop:SS}:
\begin{equation}
(U_0; 0,h_\infty|_{U_0})\oplus ((U_{ndeg})_{\lambda=1};0,\Id)\oplus ((U_{ndeg})_{\lambda\not=1};
\widetilde{b}_\infty,\widetilde{h}_\infty)\end{equation}
The last component carries a limit mixed Hodge structure which is polarized by $\widetilde{b}_\infty$, and also it
extends to a simple hermitian variation structure with $(S_{fund,ndeg})_{\lambda\not=1}$.
In such a situation, the HVS--spectrum agrees with the Hodge spectrum. The proof is
absolutely the same as in the local case, see Proposition \ref{prop:mod2}, or the original source \cite[(6.5)]{Nem2}
(or the affine polynomial case in \cite{GN}).

For the middle component both HVS and Hodge  spectra are $(\nu-1)$ copies of (1): in the HVS case see
Corollary \ref{lem:dis} as a consequence of Proposition \ref{prop:HN}(c), while for the Hodge
 case see \cite{Br} or \cite{Di}.

These two components provide the contribution from $Sp_{\HVS}(\reg{L})$. The remaining part, provided by the first summand is computed in \cite{Br}, and it is the sum in the right hand
side of the identity  of Theorem \ref{prop:finalspeccompare}.
\end{proof}

\begin{example} Recall that
$F$ is `good at infinity' if and only if $\reg{L}$ is a fibred link, that is  $n_i=1$ for all $i$,
cf. \cite[Theorem 6.1]{NeRu}.
By our result, in such a  case one has
 $Sp_{\MHS,\infty}=Sp_{\HVS}(\reg{L})$.
\end{example}

\begin{corollary}\label{cor:diff}
For any $x\in[0,1]$ one has
\[|Sp_{\HVS}(\reg{L})\cap(x,x+1)|\le |Sp_{\MHS,\infty}\cap (x,x+1)|\le |Sp_{\HVS}(\reg{L})\cap (x,x+1)|+\Irr.\]
and analogous inequality holds for $Sp\setminus[x,x+1]$.
\end{corollary}

\subsection{An example.}\label{sec:toput}
The above discussion might have been technically quite involved. We want to illustrate the occuring phenomena by investigating
one example, the Brian\c{c}on polynomial, which appeared in \cite{ACL,Br,Di,DN}
 (we remark that in \cite[Exemple 4.14]{AC} there
is a different polynomial called Brian\c{c}on polynomial,
it has different link at infinity and different irregular fibers).

The splice diagram  of the fundamental link at infinity is as follows

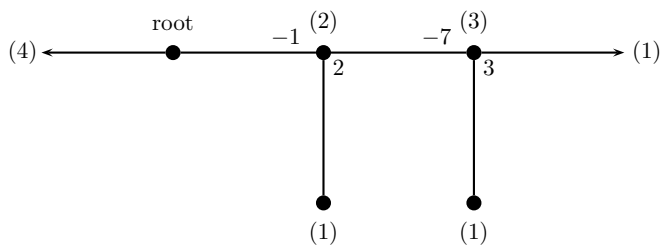
\begin{pspicture}(-7,-3)(5,1)
\pscircle[fillstyle=solid,fillcolor=black](-1,0){0.1}
\pscircle[fillstyle=solid,fillcolor=black](1,0){0.1}
\pscircle[fillstyle=solid,fillcolor=black](3,0){0.1}
\pscircle[fillstyle=solid,fillcolor=black](1,-2){0.1}
\pscircle[fillstyle=solid,fillcolor=black](3,-2){0.1}
\psline{->}(-1.1,0)(-2.75,0)
\psline{->}(3.1,0)(5,0)
\psline(-0.9,0)(0.9,0)\psline(1.1,0)(2.9,0)
\psline(1,-0.1)(1,-1.9)\psline(3,-0.1)(3,-1.9)
\rput(3.2,-0.2){\psscalebox{0.8}{$3$}}
\rput(1.2,-0.2){\psscalebox{0.8}{$2$}}
\rput(2.5,0.2){\psscalebox{0.8}{$-7$}}
\rput(0.5,0.2){\psscalebox{0.8}{$-1$}}
\rput(3,0.4){\psscalebox{0.8}{$(3)$}}
\rput(1,0.4){\psscalebox{0.8}{$(2)$}}
\rput(3,-2.4){\psscalebox{0.8}{$(1)$}}
\rput(1,-2.4){\psscalebox{0.8}{$(1)$}}
\rput(-3,0.0){\psscalebox{0.8}{$(4)$}}
\rput(5.3,0.0){\psscalebox{0.8}{$(1)$}}
\rput(-1,0.4){\psscalebox{0.8}{root}}
\end{pspicture}

Here the numbers in parentheses are the multiplicities of the vertices and arrowheads (link components).
The numbers not in parenthesis denote the weights
of corresponding edges (those omitted equal 1).
We have $n_1=4$, $n'_1=\gcd(4,6)=2$, and $n_2=n'_2=1$.

Computing the Euler characteristics of a minimal Seifert surface of $L_{fund}$
(as in \cite{EN}) we get that this surface
is a three times punctured torus  ($3$ is the number of
components of $L_{fund}$).  The rank of $H_1(Y_\infty)$ is $4$, while
the ranks of $U_0$ and $(U_{ndeg})_{\lambda=1}$ are 1.
The monodromy at infinity permutes $L_{1,1}$, $L_{1,2}$ and fixes $L_2$. The characteristic polynomial
of the monodromy on boundary components is therefore $t^2-1$. The Alexander polynomial of $L_{fund}$,
hence the characteristic polynomial of the monodromy at infinity, is
$(t^2-1)(t^2+t+1)$.

The equivariant signatures (which correspond to jumps of the Tristram--Levine signatures) of $L_{fund}$ can
be computed using \cite[Theorem 5.3 and Section 6]{Neu-spli}.
Using \cite[Theorem 5.3]{Neu-spli} for the left-most splice component we compute that $\sigma^{-}_{e^{2\pi i/3}}=-1$ and $\sigma^{-}_{e^{-2\pi i/3}}=1$
so the jumps of the Tristram--Levine signature are respectively $-2$ and $2$, in other words $p^1_{e^{2\pi i/3}}(+1)=0$, $p^1_{e^{2\pi i/3}}(-1)=1$,
$p^1_{e^{-2\pi i/3}}(+1)=1$, $p^1_{e^{-2\pi i/3}}(-1)=0$ (compare \cite[Sections 3.5, 3.6]{Br}).
On the other hand, a straightforward computation shows that the right splice component does not
contribute to the equivariant signature at all.
Hence, the non--trivial H--numbers are $p^1_{e^{2\pi i/3}}(-1)=p^1_{e^{-2\pi i/3}}(+1)=p_1^1(1)=1$.

Concluding, the spectrum at infinity is equal  to
$\{\frac 23,\frac43,\frac12,1\}$ (cf. \cite[Example 3.6(ii)]{Di}), where
$\{\frac23,1,\frac43\}$ is the contribution from  $\reg{L}$.

\subsection{The definiteness of `linking matrix'. The proof of (\ref{prop:HN})(c)}\label{sec:negative}
We wish to prove that the restriction of $S_{fund}$  on $U_{\lambda=1}$ is negative definite.
This follows from a more general combinatorial result which we now state.

Let $\Gamma$ be a rooted Eisenbud--Neumann diagram, cf. \cite{Neu1}.
For an edge we call the weight which is closer to the root vertex the \emph{near weight} and the
other one the \emph{far weight}. For any two nodes $v$ and $w$,  if the geodesics connecting $w$ and the root
vertex contains $v$ then we say that $w$ is  \emph{beyond $v$}.
We allow more than one near weight at each node to have weight different than $1$.
The linking numbers and multiplicities are determined from the diagram as in \cite[\S 10,11]{EN}.
The arrowhead vertices will be denoted by $L_1,\dots,L_\nu$, their multiplicities are
$n_1,\dots,n_\nu$.

Let ${\mathbb Q}^\nu$ be the ${\mathbb Q}$--vector space generated by $\{L_i\}_i$,
The {\it linking matrix} $\{\lk(L_i,L_j)\}_{ij}$ is defined as follows: for $i\not= j$ it is the standard
linking pairing, while the self--linking $\lk(L_i,L_i)$ is defined via the identity
$\lk(L_i,\sum_j n_jL_j)=0$. Equivalently,
\begin{equation}\label{eq:helpful}
\lk(n_iL_i,n_iL_i)=-\sum_{j\neq i}\lk(n_iL_i,n_jL_j).
\end{equation}
In particular, the null--space of the linking matrix is  at least 1--dimensional.
\begin{proposition}\label{prop:negative}
Let $\Gamma$ be a rooted connected graph with the following properties
\begin{itemize}
\item[(a)] all near weights are positive and no far weight is allowed to be zero.
\item[(b)] if the far weight at a node $v$ is negative, then all far weights of nodes beyond $v$ are also negative (this property is weaker
than negativity of edge determinants);
\item[(c)] the multiplicities of all arrowhead and non--arrowhead vertices are positive.
\end{itemize}
Then, the linking matrix $\lk(L_i,L_j)$ is negative semi--definite with 1--dimensional null--space.
\end{proposition}

\begin{proof}
We begin with a following special case.
\begin{lemma} The statement of Proposition~\ref{prop:negative} holds
if $\lk(L_i,L_j)>0$ for all $i\neq j$.
\end{lemma}
\begin{proof}
The reasoning is exactly as in \cite[\S 3]{Neu-inv}: for $L=\sum \ell_jn_jL_j$ one has
\[\lk(L,L)=\sum_{i<j}2\ell_i\ell_j\lk(n_iL_i,n_jL_j)+\sum_{i}\ell_i^2\lk(n_iL_i,n_iL_i).\]
Substituting \eqref{eq:helpful}, we get
\begin{equation}\label{eq:helpful2}
\lk(L,L)=-\sum_{i<j}(\ell_i-\ell_j)^2\lk(n_iL_i,n_jL_j).
\end{equation}
Hence $\lk(L,L)$ is zero if $\ell_1=\dots=\ell_n$, and negative otherwise.
\end{proof}

In general, if some far weights are negative, some of the linking numbers
$\lk(L_i,L_j)$ might be negative too; in these cases the proof is more involved.
\begin{lemma}\label{lem:negself}
If $\nu\geq 2$ then the self--linking number $\lk(L_i,L_i)$ is negative
for any $i$.
\end{lemma}
\begin{proof}
 For each $i$, let $v_i$ be a node supporting $L_i$, $\alpha_i$
denotes the far weight at $v_i$ and $\beta_{i1},\dots,\beta_{ik_i}$
the near weights at $v_i$, with $\beta_{i1}$ the near weight on the edge supporting $L_i$.

If  $\lk(L_i,L_j)>0$ for all $j\neq i$, then the statement follows from \eqref{eq:helpful}.
Hence, assume that $\lk(L_i,L_j)<0$  for some $j$. Assume that $L_i$ and $L_j$ are supported by
nodes $v_i$ and $v_j$ respectively (the case $v_i=v_j$ is also possible).
Let $\gamma$ be a path in $\Gamma$ joining $L_i$ to $L_j$.
Since $\lk(L_i,L_j)<0$, one of the vertices lying on $\gamma$, call it $v_\gamma$, must
have a negative weight. This, by  assumption (a), must be a far weight, hence there is a unique
$v_\gamma$ along the path with this property. Now, if $v_i$ is  beyond $v_\gamma$, then
by (b) we have $\alpha_i<0$. Otherwise, $v_i=v_\gamma$ and $\alpha_i<0$ by the definition of $v_\gamma$.
Next, let $M_i$ be the multiplicity of $v_i$, namely
\[M_i=\sum_{j} \lk(v_i,n_jL_j)=\alpha_i\beta_{i2}\dots\beta_{ik_i}n_i+\sum_{j\neq i}\lk(v_i,n_jL_j).\]
But for $j\neq i$ one has  $\lk(v_i,n_jL_j)=\beta_{i1}\lk(L_i,n_jL_j)$, hence
\begin{equation}\label{eq:estimate-lk}
-\lk(L_i,L_i)=\frac{1}{\beta_{i1}n_i}\sum_{j\neq i}\lk(v_i,n_jL_j)=\frac{M_i}{n_i\beta_{i1}}-
\frac{\alpha_i\beta_{i2}\dots\beta_{ik_i}}{\beta_{i1}}>0
\end{equation}
as $M_i>0$.
\end{proof}
\begin{corollary}\label{cor:oneortwo}
If the diagram has one or two arrowheads, then the statement of Proposition~\ref{prop:negative} holds.
\end{corollary}
\begin{proof}
Use Lemma \ref{lem:negself} and the fcat that the null--space is not trivial.
\end{proof}

The proof is based on induction via reduction of the diagram (via two operations).

\begin{definition}
Let $\Gamma$ be a rooted graph. Assume that the supporting  node $v_i$ of
the arrowhead vertex $L_i$ has the following properties: it is not the root vertex,
there is no node beyond it, $L_i$ is the unique arrowhead supported by
$v_i$. Hence, all its  adjacent vertices  except $L_i$ and another one (in the direction of the root)
are leaves.  As above, denote the valency of $v_i$
by $k_i+1$ (see the picture below).

 A \emph{collapse}
of $v_i$ is a graph $\Gamma'$ with $v_i$ replaced by an arrowhead vertex $L_i'$ with multiplicity $n_i\beta$, where $\beta=\beta_{i2}\cdots\dots\cdot\beta_{ik_i}$
and all other weights and multiplicities are unchanged.
\end{definition}

\begin{pspicture}(-7,-3)(5,1.5)
\psline[linestyle=dotted](-6.9,0)(-6.6,0)
\psline{->}(-6.5,0)(-1,0)\pscircle[fillstyle=solid,fillcolor=black](-5,0){0.1}
\pscircle[fillstyle=solid, fillcolor=black](-3,0){0.1}
\psline(-5,0)(-5,-2)\pscircle[fillstyle=solid, fillcolor=black](-5,-2){0.1}
\rput{30}(-5.0,0){\psline(0,0)(1.0,0)\psline[linestyle=dotted](1.0,0)(1.3,0)}
\rput{60}(-5.0,0){\psline(0,0)(1.0,0)\psline[linestyle=dotted](1.0,0)(1.3,0)}
\psarc[linestyle=dotted](-5.0,0){0.4}{75}{135}
\rput{150}(-5.0,0){\psline(0,0)(1.0,0)\psline[linestyle=dotted](1.0,0)(1.3,0)}
\rput(-3.4,0.15){\psscalebox{0.8}{$\alpha_i$}}
\rput(-2.6,0.15){\psscalebox{0.8}{$\beta_{i1}$}}
\rput{-35}(-3,0){\psline(0,0)(2,0)\rput{35}(0.7,0.2)
{\psscalebox{0.8}{$\beta_{i2}$}}\pscircle[fillstyle=solid,fillcolor=black](2,0){0.1}}
\rput{-60}(-3,0){\psline(0,0)(2,0)\rput{60}(0.8,0.2)
{\psscalebox{0.8}{$\beta_{i3}$}}\pscircle[fillstyle=solid,fillcolor=black](2,0){0.1}}
\rput{-135}(-3,0){\psline(0,0)(2,0)\rput{135}(0.7,0.3)
{\psscalebox{0.8}{$\beta_{ik_i}$}}\pscircle[fillstyle=solid,fillcolor=black](2,0){0.1}}
\psarc[linestyle=dotted](-3,0){2.0}{232}{292}
\rput(-1.2,0.2){\psscalebox{0.8}{$\mathbf{L_i}$}}
\rput(-0.6,0){\psscalebox{0.8}{$(n_i)$}}
\psline[linestyle=dotted](2.1,0)(2.4,0)
\psline{->}(2.5,0)(6,0)\pscircle[fillstyle=solid,fillcolor=black](4,0){0.1}
\pscircle[fillstyle=solid, fillcolor=black](4,-2){0.1}
\psline(4,0)(4,-2)
\rput{30}(4.0,0){\psline(0,0)(1.0,0)\psline[linestyle=dotted](1.0,0)(1.3,0)}
\rput{60}(4.0,0){\psline(0,0)(1.0,0)\psline[linestyle=dotted](1.0,0)(1.3,0)}
\psarc[linestyle=dotted](-5.0,0){0.4}{75}{135}
\rput{150}(4.0,0){\psline(0,0)(1.0,0)\psline[linestyle=dotted](1.0,0)(1.3,0)}
\rput(5.8,0.2){\psscalebox{0.8}{$\mathbf{L_i'}$}}
\rput(6.6,0.0){\psscalebox{0.8}{$(\beta n_i)$}}
\psline[linewidth=1.5pt]{->}(-0.2,0)(1.8,0)\rput(0.9,0.15){\psscalebox{0.8}{Collapse}}
\end{pspicture}
\begin{lemma}
The linking matrices of $\Gamma$ and $\Gamma'$ are congruent.
Moreover, if $\Gamma$ satisfies the assumptions (a), (b) and (c) of the proposition,
then so does $\Gamma'$.
\end{lemma}
\begin{proof}
We shall use the notation $\lk_\Gamma$ and $\lk_{\Gamma'}$ for the linking forms on $\Gamma$ and $\Gamma'$.

For any vertex $v$ (being  node or arrowhead), different from the deleted ones, we have $\lk_\Gamma(v,L_i)=\lk_{\Gamma'}(v,\beta L_i')$. We
claim that $\lk_{\Gamma'}(\beta L_i,\beta L_i)=\lk_{\Gamma}(L_i,L_i)$. This follows  from
 \eqref{eq:helpful} applied for $\Gamma$ and $\Gamma'$,
and from the fact that in  $\Gamma'$
the relation $\beta n_iL_i+\sum_{j\neq i}n_jL_j=0$ holds.
Thus, the linking matrix of $\Gamma$
written in a basis, $L_1,\dots,L_i,\dots,L_\nu$ is the same as the linking matrix of $\Gamma'$ in the basis
$L_1,\dots,\beta L_i',\dots,L_\nu$.
This proves the first part. As for the other part, the multiplicities of all vertices (besides the deleted ones) are
preserved. This shows that if $\Gamma$ satisfies (c), then so does $\Gamma'$, while (a) and (b) are obvious.
\end{proof}

\begin{definition}
Let $v_0$ be a node with no other node beyond it. Let $L_1,\dots,L_k$ be the arrowheads
adjacent to $v_0$ ($k\geq 2$),
denote their multiplicities by $n_1,\dots,n_k$. $v_0$ might have several adjacent leaves as well,
$\beta$ denotes the product of their near weights.
Assume that the overall number of vertices of $\Gamma$ is at least three. 

A \emph{squeeze} of $\Gamma$ is a graph arising from $\Gamma$ by replacing
two  arrowheads supported by $v_0$ (say, $L_1$ and $L_2$) by a single one, denoted by  $L_s$,
which will gain  multiplicity
\[n_s:=n_2\beta_1+n_1\beta_2\]
and the near weight  $\beta_s:=\beta_{1}\beta_2$.
\end{definition}

\begin{pspicture}(-7,-3)(7,2.5)
\psline[linestyle=dotted](-5.9,0)(-5.6,0)
\psline(-5.5,0)(-3.5,0)
\pscircle[fillstyle=solid,fillcolor=black](-3.5,0){0.1}
\rput{300}(-3.5,0){\psline(0,0)(2,0)\pscircle[fillstyle=solid,fillcolor=black](2,0){0.1}}
\rput{210}(-3.5,0){\psline(0,0)(2,0)\pscircle[fillstyle=solid,fillcolor=black](2,0){0.1}}
\psarc[linestyle=dotted](-3.5,0){1.0}{215}{295}
\rput(-3.7,0.15){\psscalebox{0.8}{$\alpha$}}
%
\rput{330}(-3.5,0){\psline{->}(0,0)(2,0)\rput{30}(2.3,0){\psscalebox{0.8}{$(n_1)$}}
\rput{30}(0.6,0.2){\psscalebox{0.8}{$\beta_1$}}}
\rput{10}(-3.5,0){\psline{->}(0,0)(2,0)\rput{-10}(2.3,0){\psscalebox{0.8}{$(n_{2})$}}
\rput{-10}(0.6,0.2){\psscalebox{0.8}{$\beta_{2}$}}}
\psarc[linestyle=dotted](-3.5,0){1.0}{15}{95}
\rput{100}(-3.5,0){\psline{->}(0,0)(2,0)\rput{-100}(2.3,0){\psscalebox{0.8}{$(n_k)$}}
\rput{-100}(0.6,0.2){\psscalebox{0.8}{$\beta_k$}}}
\psline[linestyle=dotted](2.1,0)(2.4,0)\psline(2.5,0)(4,0)
\pscircle[fillstyle=solid,fillcolor=black](4,0){0.1}
\rput{300}(4,0){\psline(0,0)(2,0)\pscircle[fillstyle=solid,fillcolor=black](2,0){0.1}}
\rput{210}(4,0){\psline(0,0)(2,0)\pscircle[fillstyle=solid,fillcolor=black](2,0){0.1}}
\psarc[linestyle=dotted](4,0){1.0}{215}{295}
\rput(3.7,0.15){\psscalebox{0.8}{$\alpha$}}
\rput{10}(4,0){\psline{->}(0,0)(2,0)\rput{-10}(2.3,0){\psscalebox{0.8}{$(n_s)$}}
\rput{-10}(0.6,0.2){\psscalebox{0.8}{$\beta_s$}}}
\rput{50}(4,0){\psline{->}(0,0)(2,0)\rput{-50}(2.3,0){\psscalebox{0.8}{$(n_{3})$}}
\rput{-50}(0.6,0.2){\psscalebox{0.8}{$\beta_{3}$}}}
\psarc[linestyle=dotted](4,0){1.0}{55}{95}
\rput{100}(4,0){\psline{->}(0,0)(2,0)\rput{-100}(2.3,0){\psscalebox{0.8}{$(n_k)$}}
\rput{-100}(0.6,0.2){\psscalebox{0.8}{$\beta_k$}}}
\psline[linewidth=1.5pt]{->}(-1.0,0)(1.5,0)\rput(0.25,0.15){\psscalebox{0.8}{Squeeze}}
\end{pspicture}

\begin{lemma}\label{lem:collapse}
Let $\Gamma'$ be a squeeze of arrowhead $L_1$ and $L_{2}$ from $\Gamma$.
If $\Gamma$ satisfies the assumptions (a), (b) and (c) of the proposition,
then so does $\Gamma'$. Moreover, the rational linking matrix of $\Gamma$ is a
direct  sum of the linking matrix of $\Gamma'$ and a negative definite 1--dimensional matrix.
\end{lemma}
\begin{proof}

As for the first part we observe that $\beta_s$ and $n_s$ were chosen in such a way that all multiplicities of vertices are preserved. Moreover,
by construction we have
\begin{equation}\label{eq:GeqGp}\begin{array}{ll}
\lk_{\Gamma'}(L_i,L_j)=\lk_{\Gamma}(L_i,L_j) & \text{ if }\{i,j\}\cap\{1,2\}=\emptyset\text{ and }i\neq j,\\
\lk_{\Gamma'}(n_sL_s,L_j)=\lk_{\Gamma}(n_1L_1+n_2L_2,L_j) & \text{ if } \ j\geq 3.\end{array}
\end{equation}
We claim $\lk_{\Gamma'}(L_j,L_j)=\lk_{\Gamma}(L_j,L_j)$ for $j\geq 3$. Indeed, this follows from
(\ref{eq:GeqGp}) and (\ref{eq:helpful}) applied for both graphs.
Now, let us define
\[\Lambda_1=\beta_1L_1-\beta_2L_2\ \ \ \mbox{and} \ \ \ \Lambda_2=xL_1+yL_2,\]
where the rational numbers $x$ and $y$ will be determined later. By definition,
\[ \lk_\Gamma (\Lambda_1, L_j)=0 \ \ \ \mbox{for any $j\geq 3$}.\]
The self-linking of $\Lambda_1$ is equal to
\[\lk_\Gamma(\Lambda_1,\Lambda_1)=\beta_1^2\lk_\Gamma(L_1,L_1)+
\beta_2^2\lk_\Gamma(L_2,L_2)-2\alpha\beta\beta_1\beta_2\dots\beta_k.\]
If $\alpha>0$, then the above expression is negative, because
$\lk_\Gamma(L_1,L_1)$ and $\lk_\Gamma(L_2,L_2)$ are negative  by Lemma~\ref{lem:negself}. If $\alpha<0$,
we use \eqref{eq:estimate-lk} to show that
$\lk_\Gamma(\Lambda_1,\Lambda_1)=-M_{v_0}(\frac{\beta_1}{n_1}+\frac{\beta_2}{n_2})<0$, because
the multiplicity of $M_{v_0}$ is positive. Hence $\lk_\Gamma(\Lambda_1,\Lambda_1)<0$ always.

Since $\lk_\Gamma(\Lambda_1,\Lambda_1)<0$ and $\lk_\Gamma(L_1,L_2)\not=0$,
there exist $x$ and $y$ such that $\lk_\Gamma(\Lambda_2,\Lambda_1)=0$
and $\Lambda_1,\Lambda_2$ are linearly independent.
Such $x$ and $y$ are determined up to a multiplicative constant. To choose it observe that
\[\lk_\Gamma(\Lambda_2,L_j)=x\lk_\Gamma(L_1,L_j)+y\lk_\Gamma(L_2,L_j)=
\left(x+y\frac{\beta_1}{\beta_2}\right)\lk_\Gamma(L_1,L_j)\]
and $\lk_{\Gamma'}(L_s,L_j)=\frac{1}{\beta_{2}}\lk_{\Gamma}(L_1,L_j)$. We chose the rational numbers
 $x$ and $y$ so that $x+y\frac{\beta_1}{\beta_2}=\frac{1}{\beta_2}$.
Then, we have for all $j\ge 3$
\begin{equation}\label{eq:linkequal}
\lk_\Gamma(\Lambda_2,L_j)=\lk_{\Gamma'}(L_s,L_j).
\end{equation}

Finally, we show that $\lk_\Gamma(\Lambda_2,\Lambda_2)=\lk_{\Gamma'}(L_s,L_s)$.
This is done as follows. First,
on $\Gamma$ we have the relation $n_1L_1+n_2L_2+\sum n_jL_j=0$, which can be rewritten as
\[\lambda_1\Lambda_1+\lambda_2\Lambda_2+\sum_{j\ge 3}n_jL_j=0\]
for some $\lambda_1$ and $\lambda_2$.
On the other hand, on $\Gamma'$ we have $n_sL_s+\sum_{j\ge 3}n_jL_j=0$. Now taking the linking numbers with $L_r$ for some $r\ge 3$ we obtain
\[0=\sum_{j\ge 3}n_r\lk_\Gamma(L_r,L_j)+\lambda_2\lk_\Gamma(L_r,\Lambda_2)=\sum_{j\ge 3}n_r\lk_{\Gamma'}(L_r,L_j)+n_s\lk_{\Gamma'}(L_r,L_s).\]
Now by \eqref{eq:GeqGp}, since $r\ge 3$ the above equation simplifies to
\[\lambda_2\lk_{\Gamma}(L_r,\Lambda_2)=n_s\lk_{\Gamma'}(L_r,L_s).\]
From \eqref{eq:linkequal} and $\lk_{\Gamma'}(L_r,L_s)\not=0$
it follows that $n_s=\lambda_2$. But then we have
\[\lk_\Gamma(\Lambda_2,\lambda_2\Lambda_2)=-\sum_{j\ge 3}\lk_\Gamma(\Lambda_2,n_jL_j)=-\sum_{j\ge 3}\lk_{\Gamma'}(L_s,n_jL_j)=\lk_{\Gamma'}(L_s,n_sL_s).\]
As $n_s=\lambda_2$ we conclude that $\lk(\Lambda_2,\Lambda_2)=\lk(L_s,L_s)$.
Hence the linking form on $\Gamma$ restricted to $\Lambda_2,L_3,\dots,L_n$ is the same as the linking form on $\Gamma'$
written in basis $L_s,L_3,\dots,L_n$, while the element $\Lambda_1$ splits out completely as an orthogonal summand.
\end{proof}

\smallskip
\noindent\emph{Finishing the proof of Proposition~\ref{prop:negative}.}
By applying collapses and squeezes to $\Gamma$ we end up
with a diagram, for which no further collapse or squeeze is possible.
This  diagram  has one or two arrowheads and we conclude the proof by Corollary~\ref{cor:oneortwo}.
\end{proof}
\begin{remark}
If we assume that the multiplicities of nodes of $\Gamma$ are only non-negative (not just positive),
we can still prove semidefiniteness of the linking matrix, possibly with higher dimensional null--space.
We omit the details.
\end{remark}

\section{Semicontinuity results}\label{S:results}
Now we are ready to prove various semicontinuity results.
In subsection~\ref{ss:sem-local-case} we recover (in a slightly weaker form) the
classical semicontinuity results valid in the local case
of algebraic plane curve singularities (classically proved by Varchenko in
\cite{Var}, see also \cite{St}). Next, in  \ref{ss:sem-inf-2}, we analyse the behavious of
spectra under a degeneration of affine plane curves
in the spirit of \cite{NS}. Finally,
we consider an affine plane curve, and we
relate the spectrum of a curve at infinity with the spectra of its singularities, see \ref{ss:sem-inf}.
This type of comparison is unknown in Hodge theory.

\subsection{Semicontinuity of the local singularity spectrum}\label{ss:sem-local-case}
Recall that in the local case $Sp_{\MHS}=Sp_{\HVS}$ (cf. \ref{prop:mod2}), which will be denoted just by $Sp$.

Let us consider now the following situation. Let $f_t(x,y)$ be a smooth
 family of holomorphic functions in two local coordinates depending on a local parameter $t$.  Assume that
$f_0(x,y)=0$ has an isolated singularity at the origin. Let us introduce the following notation.
\begin{itemize}
\item $B$ is a
small ball centered at the origin such that $f_0^{-1}(0)$ is transverse to $\partial B$ and
$f_0(z)/|f_0(z)|\colon\partial B\setminus f_0^{-1}(0)\to S^1$  is a Milnor fibration;
\item $L_0=f^{-1}_0(0)\cap \partial B$ is the link of $f_0$ at $0$, and  $Sp_0$ the spectrum of the link;
\item $t\not=0$ and $|t|$ is  sufficiently small so that  $f_{t}^{-1}(0)\cap \partial B$ is a transversal
intersection, and this link is isotopic in  $\partial B$ to $L_0$;
\item $C=f_{t}^{-1}(0)\cap B$;
\item $z_1,\dots,z_k$ are singular points of $C$, $\LS_1,\dots,\LS_k$ the corresponding local links
 of these singularities,
and $Sp_1,\dots,Sp_k$ are the spectra of $\LS_1,\dots,\LS_k$.
\end{itemize}


\begin{proposition}\label{prop:localsemic}
Fix  $x\in [0,1]$ such that $e^{2\pi ix}$ is not a root of the Alexander polynomial of $L_0$. Then
\begin{equation}\label{eq:simplesemic}
\begin{split}
|Sp_0\cap (x,x+1)|&\ge \sum_{j}|Sp_j\cap (x,x+1)|\\
|Sp_0\setminus [x,x+1]|&\ge \sum_{j}|Sp_j\setminus [x,x+1]|.
\end{split}
\end{equation}
\end{proposition}
\begin{proof}
Assume $x\neq 0,1$. We shall prove only the first inequality, the second is completely analogous (in Section~\ref{S:morseplane} all
inequalities are given in pairs, the first one we use to prove results about $Sp\cap(x,x+1)$, the other one to prove results about $Sp\setminus[x,x+1]$).
As $\LS_j$ is an algebraic links, $\mu_j$ is the degree of the
Alexander polynomial of $\LS_j$. Hence
$\mu_j-\sigma_{\LS_j}(\zeta)+n_{\LS_j}(\zeta)\ge 2|Sp_j\cap(x,x+1)|$ by Corollary \ref{cor:univ}.

By assumption $n_{L_0}(\zeta)=0$.
Since $L_0$ is also an algebraic link,
$1-\chi(C_{smooth})$ is the degree of the Alexander polynomial of $L_0$. Thus, again by Corollary~\ref{cor:univ}, one gets
$-\sigma_{L_0}(\zeta)+(1-\chi(C_{smooth}))=2|Sp_0\cap(x,x+1)|$.
Then we conclude by the inequality   \eqref{eq:diff-r0}.

If $x\in\{0,1\}$ the assumption that $e^{2\pi ix}$ is not a root of the Alexander polynomial means that
$L_0$ is a knot and $|Sp_0\cap(0,1)|=|Sp_0\cap(1,2)|$ is the delta invariant $\delta_0$.
For any singularity link, hence for $L_j^{sing}$ too,  $\delta_j=|Sp\cap (0,1]|\geq |Sp_j\cap(0,1)|$. Hence
the statement follows from  $\delta_0\ge\sum\delta_j$.
\end{proof}


\subsection{Semicontinuity of  spectrum at infinity of families of affine curves}\label{ss:sem-inf-2}
The methods described in this paper allow us also to prove the results on
semicontinuity of the spectrum at infinity in the sense of \cite{NS}.

Let $F_t\colon\mathbb{C}^2\to\mathbb{C}$
be a smooth family of polynomials with a local deformation parameter $t$.
Let $Sp_t$ be corresponding spectra of $MHS$ at infinity and $\Irr_t$ be the
irregularity of the link at infinity $L^\infty_{t,reg}$ associated with $F_t$.

Note that over  a small punctured  disc $D^*\ni t$ the spectrum  $Sp_t$ is constant.

\begin{theorem}\label{prop:semic}
Fix $x\in[0,1]$ such that $\{x,x+1\}\cap Sp_t=\emptyset$ for $t\in D^*$.
Then
\[|Sp_t\cap(x,x+1)|+\Irr_t\ge |Sp_0\cap(x,x+1)|\]
and the same statement holds for $Sp\setminus [x,x+1]$ instead of $Sp\cap (x,x+1)$.
\end{theorem}
\begin{proof}
Again we shall assume that $x$ is not an integer, otherwise we use exactly
the same reduction argument as at the end of  the proof of \ref{thm:semiinf}.
We write $\zeta:=e^{2 \pi ix}\in S^1\setminus \{1\}$.

Let us choose  $c$ such that $C_0=F_0^{-1}(c)$ is smooth and regular at infinity.
Furthermore, choose  $\xi$ and $r_0$ such that $S^3(\xi,r_0)\cap C_0$
is the regular link of $F_0$ at infinity, denoted by  $L_0$.
By openness of trasversality condition, there exist $D$,
an open neighbourhood of $0$,  and $W_c$, an open neighbourhood of $c$, such
that for any $w\in W_c$ and $t\in D$, the intersection $F_{t}^{-1}(w)\cap S^3(\xi,r_0)$
is transverse and isotopic to $L_0$. Let us take any $t\in D^*$ and choose $w\in W_c$
such that $C_t=F_t^{-1}(w)$ is smooth and regular at infinity.
Finally, choose  $r_t$ such that $L_t:=S^3(\xi,r_t)\cap C_t$ is the regular link at infinity  of $C_t$.

Since $C_t\cap B(\xi,r_0)$ is isotopic to $C_0\cap B(\xi,r_0)$,
by Corollary~\ref{cor:smooth-morse}   we get
\[-\sigma_{L_t}(\zeta)+n_{L_t}(\zeta)+1-\chi(C_t\cap B(\xi,r_t))\ge -\sigma_{L_0}(\zeta)+n_{L_0}(\zeta)+1-\chi(C_0\cap B(\xi,r_0)).\]
By assumption, $\zeta$ is not a root of $\Delta_{\Irr_t}(L_t)$.
Hence, applying \ref{prop:degcp} for $F_t$ and $F_0$ we obtain
\[-\sigma_{L_t}+\deg\Delta_{\Irr_t}(L_t)+2\Irr_t\ge -\sigma_{L_0}(\zeta)+\tilde{n}_{L_0}(\zeta)+\deg\Delta_{\Irr_0}+2\Irr_0.\]
Then Corollary~\ref{cor:univ} implies
\[|Sp_{\HVS}(L_t)\cap (x,x+1)|+\Irr_t\ge |Sp_{\HVS}(L_0)\cap (x,x+1)|+\Irr_0.\]
Finally, Corollary~\ref{cor:diff} provides the result.

To show the statement for $Sp\setminus[x,x+1]$ we use the same argument.
\end{proof}

\subsection{Spectrum at infinity of a singular curve}\label{ss:sem-inf}
Let $C\subset\mathbb{C}^2$ be an irreducible plane algebraic curve given by zero set of a reduced polynomial $F$.
Let $z_1,\dots,z_k$ be its singular points and  $Sp_1$,\dots,$Sp_k$ their
(Hodge or HVS) spectra.
Let $Sp_\infty$ be the Hodge--spectrum of $F$ at infinity. Similarly, let  $\reg{L}$ be the
regular link of $F$ at infinity, $Sp_{\HVS}(\reg{L})$ its HVS--spectrum  and $\Irr$ be as defined in \eqref{eq:Ndef}.
\begin{theorem}\label{thm:semiinf}
With the above notations, for all $x\in[0,1]$
such that $e^{2\pi i x}$ is not
a root of the Alexander polynomial of $\reg{L}$
we have
\begin{equation}\label{eq:finiteinf1}\begin{split}
\left\vert Sp_{\HVS}(\reg{L})\cap(x,x+1)\right\vert+\Irr\ge \sum_{j}\left\vert Sp_j\cap (x,x+1)\right|\\
\left\vert Sp_\infty\cap(x,x+1)\right\vert+\Irr\ge \sum_{j}\left\vert Sp_j\cap (x,x+1)\right|.
\end{split}\end{equation}
Moreover, the analogous statement holds if we replace $Sp\cap (x,x+1)$ by
$Sp\setminus [x,x+1]$.
\end{theorem}
\noindent
In the good case, if the regular link at infinity is fibred (e.g. if it is a knot), then $\Irr=0$ and the
second inequality of \eqref{eq:finiteinf1} takes the form
$|Sp_\infty\cap (x,x+1)|\ge\sum_{k=1}^n\left|Sp_k(x,x+1)\right|$.
\begin{proof}
First we assume that $x\in(0,1)$. We focus on the case $Sp\cap(x,x+1)$ the case $Sp\setminus[x,x+1]$ is analogous.

If $C$ is regular at infinity,
the inequality \eqref{eq:diff-r0} reads as
\begin{equation}\label{eq:prstep1}
-\sigma_{\reg{L}}(\zeta)+n_{\reg{L}}(\zeta)+(1-\chi(C_{smooth}))\ge \sum_{j}(-\sigma_j(\zeta)+n_j(\zeta)+\mu_j),
\end{equation}
where $C_{smooth}$ is the smoothing of $C$.
Since each link $\LS_j$ is algebraic,  $-\sigma_j(\zeta)+n_j(\zeta)+\mu_j
\ge 2|Sp_j\cap(x,x+1)|$. On the other hand,
by Proposition~\ref{prop:degcp} we get
\begin{equation}\label{eq:prstep2}
1-\chi(C_{smooth})=\Irr+\deg\Delta_{\Irr}.
\end{equation}
By Lemma~\ref{lem:Lreg}(d)  $\cp_{\reg{L}}$ has no roots outside the unit circle, hence
Corollary~\ref{cor:univ} applies. Since $\zeta$ is not a root of $\cp_{\reg{L}}$, $\nn(\zeta)=0$,
hence
\begin{equation}\label{eq:prstep3}
-\sigma_{\reg{L}}(\zeta)+\deg\cp_{\reg{L}}=2|Sp_{\reg{L}}\cap (x,x+1)|,
\end{equation}
and $n(\zeta)=\Irr$ (see \ref{ss:Lfund}). Then
\eqref{eq:prstep1}, \eqref{eq:prstep2} and \eqref{eq:prstep3} prove the statement in this case.

If $C$ is not regular at infinity, we argue as follows. We take an $r_0$ such that
$C\cap S^3(\xi,r_0)$ is the link of $C$ at infinity, denoted by  $L_C$.
Then \eqref{eq:diff-r0} yields
\begin{equation}\label{eq:irreg-1}
-\sigma_{L_C}(\zeta)+n_{L_C}(\zeta)+(1-\chi(C^{r_0}_{smooth}))\ge \sum_{j}(-\sigma_j(\zeta)+n_j(\zeta)+\mu_j),
\end{equation}
where
$C^{r_0}_{smooth}:=C_\varepsilon\cap B(\xi,r_0)$ is the smoothing of $C$ in $B(\xi,r_0)$. Here
$C_\varepsilon:= F^{-1}(\varepsilon)$ is smooth and regular at infinity
(for $\varepsilon$ non-zero and sufficiently small).
Moreover, we can assume that
the links $C\cap S^3(\xi,r_0)$ and  $C_\varepsilon\cap S^3(\xi,r_0)$ are isotopic.
Let $r_1$ be such that $C_\varepsilon\cap S^3(\xi,r_1)$ is the regular link of $F$ at infinity.
Corollary~\ref{cor:smooth-morse} applied to $C_\varepsilon$ yields
\begin{equation}\label{eq:irreg-2}
-\sigma_{\reg{L}}(\zeta)+n_{\reg{L}}(\zeta)-(-\sigma_{L_C}(\zeta)+n_{L_C}(\zeta))\ge\chi(C_{01}),
\end{equation}
where $C_{01}=C_\varepsilon\cap (B(\xi,r_1)\setminus B(\xi,r_0))$.
 \eqref{eq:irreg-1} and \eqref{eq:irreg-2} combined yields
\[-\sigma_{\reg{L}}(\zeta)+n_{\reg{L}}(\zeta)+(1-\chi(C_\varepsilon))\ge \sum_{j}(-\sigma_j(\zeta)+n_j(\zeta)+\mu_j).\]
This inequality  is identical to \eqref{eq:prstep1} and we proceed further as in the previous case.

Assume that $x=0$. Then, by the assumption, 1 is not a root of the
Alexander polynomial of $\reg{L}$, hence $\reg{L}$ is a knot (because $U_{\lambda=1}$ is trivial, but  its dimension
is $\nu-1$ by  Proposition~\ref{prop:SS}). Therefore the link at infinity is good, $\Irr=0$ and
$Sp_{\HVS}(\reg{L})=Sp_\infty$.

For $\theta>0$ sufficiently small $|Sp_\infty\cap (0,1)|=
|Sp_\infty\cap(\theta,1+\theta)|$ (because $1\not\in Sp_{\HVS}(\reg{L})=Sp_\infty$).
On the other hand, in the local case,
$|Sp_j\cap(0,1)|\le |Sp_j\cap(\theta,1+\theta)|$, hence
 the statement follows from the  case $x\in(0,1)$.

The case $x=1$ follows by the same argument with $\theta<0$.
\end{proof}

\end{document}